\documentclass[12pt, a4paper]{amsart}
\usepackage[english]{babel}
 \usepackage[utf8x]{inputenc} 
\usepackage{amsmath}
\usepackage{cmll}
\usepackage{amssymb}
\usepackage{amsthm}
\usepackage{xcolor}
\usepackage{bbm}
\usepackage{mathrsfs}
\usepackage[all]{xy}
\usepackage{indentfirst}
\usepackage{fancyhdr}
\usepackage{newlfont}
\usepackage{setspace}
\usepackage{verbatim}
\usepackage{hyperref} 
\usepackage{bbm}

\usepackage[a4paper,top=3.2cm,bottom=3.2cm, left=2.5cm,right=2.5cm]{geometry}

\renewcommand{\H}{\mathbb{H}}
\newcommand{\B}{\mathbb{B}}

\newcommand{\G}{\mathbb{G}}
\newcommand{\K}{\mathbb{K}}
\newcommand{\M}{\mathbb{M}}
\newcommand{\N}{\mathbb{N}}

\newcommand{\R}{\mathbb{R}}

\newcommand{\T}{\mathbb{T}}
\newcommand{\U}{\mathbb{U}}
\newcommand{\V}{\mathbb{V}}
\newcommand{\W}{\mathbb{W}}

\newcommand{\elle}{\mathbb{L}}

\newcommand{\cF}{\mathcal{F}}

\newcommand{\cH}{\mathcal{H}}

\newcommand{\cL}{\mathcal{L}}

\newcommand{\cP}{\mathcal{P}}

\newcommand{\cS}{\mathcal{S}}

\newcommand{\cV}{\mathcal{V}}

\newcommand{\cp}{\mathfrak{p}}

\newcommand{\cs}{\mathfrak s}

\newcommand{\ep}{\varepsilon}

\newcommand{\sm}{\setminus}

\newcommand{\ccS}{\mathfrak{S}}

\newcommand{\diam}{\mbox{\rm diam}}

\newcommand{\spn}{\mbox{\rm span}}

\newcommand{\n}{\mathrm n}

\newcommand{\lan}{\langle}
\newcommand{\ran}{\rangle}
\newcommand{\res}{\mbox{\LARGE{$\llcorner$}}}

\newcommand{\der}{\partial}

\newcommand{\beqas}{\begin{eqnarray*}}
\newcommand{\eeqas}{\end{eqnarray*}}
\newcommand{\beqa}{\begin{eqnarray}}
\newcommand{\eeqa}{\end{eqnarray}}
\newcommand{\beq}{\begin{equation}}
\newcommand{\eeq}{\end{equation}}
\newcommand{\bce}{\begin{center}}
\newcommand{\ece}{\end{center}}

\newcommand{\pa}[1]{\left( #1 \right)}               
\newcommand{\set}[1]{\left\{ #1 \right\}}            

\newtheorem{The}{Theorem}[section]
\newtheorem{Lem}[The]{Lemma}
\newtheorem{Def}[The]{Definition}
\newtheorem{Pro}[The]{Proposition}
\newtheorem{Cor}[The]{Corollary}
\newtheorem{Exa}[The]{Example}
\newtheorem{Con}{Conjecure}

\newcommand{\bt}{\begin{The}}
\newcommand{\et}{\end{The}}
\newcommand{\bl}{\begin{Lem}}
\newcommand{\el}{\end{Lem}}
\newcommand{\bd}{\begin{Def}\rm}
\newcommand{\ed}{\end{Def}}
\newcommand{\br}{\begin{Rem}\rm}
\newcommand{\er}{\end{Rem}}
\newcommand{\bpr}{\begin{Pro}}
\newcommand{\epr}{\end{Pro}}
\newcommand{\bc}{\begin{Cor}}
\newcommand{\ec}{\end{Cor}}
\newcommand{\bj}{\begin{Con}}
\newcommand{\ej}{\end{Con}}
\newcommand{\bex}{\begin{Exa}}
\newcommand{\eex}{\end{Exa}}

\newcommand{\bV}{\mathbf{V}}

\newcommand{\bW}{\mathbf{W}}
\newcommand{\bU}{\mathbf{U}}
\newcommand{\IL}{\mathcal{I}\hskip0.0001mm\mathcal{L}}
\renewcommand{\graph}{\mathrm{graph}}

\newtheorem{teo}{Theorem}[section]
\newtheorem{prop}[teo]{Proposition}

\newtheorem{lem}[teo]{Lemma}

\theoremstyle{definition}

\newtheorem{deff}[teo]{Definition}

\newtheorem{Remark}[teo]{Remark}

\allowdisplaybreaks

\begin{document}

\title[Area of intrinsic graphs in homogeneous groups]{Area of intrinsic graphs in homogeneous groups}

\author{Francesca Corni}
\address{Francesca Corni: Dipartimento di Matematica\\ Universit\`a di Bologna\\ Piazza di Porta S.Donato 5\\ 40126, Bologna, Italy}
\email{francesca.corni3@unibo.it}
\author{Valentino Magnani}
\address{Valentino Magnani: Dipartimento di Matematica\\ Universit\`a di Pisa\\ Largo Bruno Pontecorvo 5\\ 56127, Pisa, Italy}
\email{valentino.magnani@unipi.it}

\begin{abstract}
We establish an area formula for the spherical measure of intrinsic graphs of any codimension in homogeneous groups. Our approach relies on the assumption that the map defining the intrinsic graph is continuously intrinsically differentiable. The main novelty is a notion of Jacobian defined using an auxiliary scalar product.
\end{abstract}

\thanks{F.C. is partially supported by INDAM--GNAMPA-2023 project: {\it Equazioni completamente non lineari locali e non locali}. V.M. is partially supported by the APRISE - {\em Analysis and Probability in Science} project, funded by the University of Pisa, grant PRA 2022 85, 
by PRIN 2022PJ9EFL {\em Geometric Measure Theory: Structure of Singular Measures, Regularity Theory and Applications in the Calculus of Variations}, funded by the European Union--NextGenerationEU", CUP:E53D23005860006, and by
the MIUR Excellence Department Project awarded to the Department of Mathematics, University of Pisa, CUP I57G22000700001.}


\date{\today}


\maketitle

\tableofcontents

\section{Introduction}

In the last two decades, the study of Geometric Measure Theory in a non-Euclidean framework has attracted a great deal of attention and interest. While on one side the results in this setting also include the classical ones in Euclidean spaces, more general approaches are necessary to work with both a noncommutative group operation and a distance that is not bi-Lipschitz equivalent to the Euclidean distance. 

In the present work, we consider noncommutative homogeneous groups, which arise from Harmonic Analysis, \cite{Fol75,SteinICM1977,FS82,Stein93}.
Our aim is to find an area formula for the spherical measure of an intrinsically regular submanifold in an arbitrary homogeneous group. The intrinsic regularity is motivated by the theory of sets of finite perimeter in Carnot groups, \cite{FSSC01,FSSC5,AKLD2009TangSpGro,Mag31,DonLDMV2022}. 
We focus on intrinsic graphs, which define the notion of intrinsic rectifiable set. Their geometric and regularity properties have been studied in several papers, including \cite{FSSC06, ASCV06, AS2009, FMS14, FranchiSerapioni2016IntrLip, SerraCassano2016}, but the list could be enlarged.

An important aspect is that intrinsic regular graphs, defined below, can be in general only H\"older continuous with respect to the Euclidean distance. They need not even be rectifiable in the Euclidean sense, and their low regularity is the fundamental difficulty, since the standard approaches from Geometric Measure Theory in Euclidean spaces do not apply. For this reason, combining the low Euclidean regularity of intrinsic graphs and fine notions of {\em intrinsic regularity} represents an intriguing problem, which in recent years has led to important contributions, applications, and challenging questions. To give an idea of the intense research activity on this subject, we limit ourselves to mentioning only some related papers,
\cite{ASCV06, FSSC6, BSC10DistSol, FMS14, BigCarSerC2015, FranchiSerapioni2016IntrLip, NY18, CFO2019beta, CFO2019bound, FOR20, JNGV21, JNGV22, DFO2022, ADDD22, Vit22, CorMag22, AntMer22, Mer22, Mer23, NY22, CLY22, DDF22, DonMag2023, ADDDLD24, AntYou24pr, CorMag25, DiMJulNGoloVit25, FraPan25pr, HakHeiIko25pr}, being aware that this list is far from being complete. 

The natural notion of area, which takes into account the geometry of the group, is given by the spherical measure, the Hausdorff measure, or the centered Hausdorff measure, constructed by a fixed homogeneous distance. Then several versions of the area formula are also available for higher codimensional {\em intrinsic regular graphs} or sets of finite perimeter
\cite{FSSC01,FSSC03,FSSC5,FSSC6,CMPSC14,FSSC15,SerraCassano2016, Mag31,CorMag22,JNGV22,AntMer22,ADDDLD24}.
To stress the difference between rectifiability in Carnot groups and the classical notion of rectifiability in Euclidean spaces, we mention that a set rectifiable with respect to the former definition is not necessarily rectifiable in the classical sense of Geometric Measure Theory in Euclidean spaces, see \cite{KirSer04,MagMir26inpr}.

We consider a {\em homogeneous group} $\G$ equipped with a homogeneous distance $d$ and a left invariant graded Riemannian metric. We take a couple of complementary subgroups $(\W,\V)$ of $\G$ and a map $\phi: A \to \V$, with $A \subset \W$ open, see Section~\ref{sect:NotFactHom}. We assume that $\phi$ is intrinsically differentiable at every point of $A$ (Definition~\ref{d:intrinsicDiff}) and that the intrinsic differential $A \ni w \to d\phi_w$ is a continuous map. 
The main result of the paper is an area formula for the spherical measure of the set 
$$
\Sigma=\graph\{ \phi \}= \{ w \phi(w) : w \in A \}
$$
with respect to an arbitrary homogeneous distance $d$. The previous set is the {\em intrinsic graph} of $\phi$. Its associated graph map is denoted by $\Phi:A \to \Sigma$, $\Phi(w)=w \phi(w)$. 

First of all, it is important to stress that the point of any area formula is to provide an effective way to compute the Hausdorff (or spherical) measure of a parametrized set. To reach this objective, an explicit formula for the Jacobian of the parametrization is necessary. For a very general notion of Jacobian, that may not be easy to compute, an area formula can be stated in metric spaces with no differentiable structure, \cite{Mag2011cm}.
When the metric space possesses a rich (metric and algebraic) structure, as a homogeneous group, then finding a more explicit notion of Jacobian becomes the first important question. The notion should involve a ``suitable differential'' of the parametrization and possibly an explicit formula in terms of suitable local coordinates. 

The central novelty of this work exactly lies in a new and explicit notion of Jacobian for intrinsic graph mappings in general homogeneous groups. Let us denote by $m$ the topological dimension of the subgroup $\W$.
Then the intrinsic Jacobian of the graph mapping $\Phi$ at $w \in A$ is
\begin{equation}\label{intro:uno}
	J\Phi(w)=\frac{\mathcal{H}_{|\cdot|}^m(G(d\phi(w))(B))}{\mathcal{H}^m_{|\cdot|}(B)},
\end{equation}
where $G(d\phi(w)): \W \to \G$ is the graph map associated with the intrinsic differential $d \phi_w$, $\mathcal{H}_{|\cdot|}^m$ is the Hausdorff measure corresponding to the fixed scalar product on $\G$ and $B \subset \W$ is a Borel set with positive measure. 
The point of \eqref{intro:uno} is that it can be explicitly computed once we have the intrinsic differential. It is also in perfect analogy with the Jacobian for homogeneous homomorphisms between stratified groups, \cite[Definition~10]{Mag}. Let us point out that stratified groups, also known as Carnot groups, are a special case of homogeneous groups. Indeed, although many motivations come from Carnot groups, our results are achieved in the general framework of homogeneous groups.

Broadly speaking, the idea behind the definition \eqref{intro:uno} is to perform two ``linearization processes'' on the intrinsic graph. The first one is provided by the intrinsic differentiability of $\phi$, hence leading to the graph map $G(d\phi(w))$ of the intrinsic differential. The second one is a
``Euclidean linearization'' of $G(d\phi(w))$ arising from the classical area formula. 

However, a first difficulty in the above definition of Jacobian emerges, since the graph map $G(d\phi(w))$ is not linear,
hence the independence of \eqref{intro:uno} from the set $B$ is not straightforward.
The proof of this invariance with respect to $B$ is given in Proposition~\ref{propjac}
and it relies on a highly nontrivial algebraic lemma, see Lemma~\ref{lm:MVPsubgroups}.
The important role of this technical lemma will be discussed below. 
Using \eqref{intro:uno}, we can introduce the natural measure associated with $\Sigma$ as 
\begin{equation}
\mu(B)= \int_{\Phi^{-1}(B)} J\Phi(w)\  d \mathcal{H}_{|\cdot|}^m (w),
\end{equation}
for every Borel set $B \subset \G$. Thus, the {\em measure-theoretic area formula} of \cite[Theorem~5.7]{LecMag22} shows that our problem boils down to computing the spherical Federer density $\cs^M(\mu,x)$ of the measure $\mu$ at $x$ (Definition~\ref{def:sphefederer}), where $x\in \Sigma$ and
$M$ is the Hausdorff dimension of the subgroup $\W$. The Federer density was first introduced in \cite{Mag30}, for Borel regular measures in metric spaces.
In our context, the computation of this density is not easy, since it is related to the geometry of both the intrinsic graph and the group.

The crucial tool to determine the Federer density is the upper blow-up of Theorem~\ref{theorem:ubu},
which is the first key result of this paper. We state it in Theorem~\ref{theorem:ubuIntro}.
For more information about the notions appearing in this theorem, we refer to Sections~\ref{sect:compendium} and \ref{sect:NotFactHom}. 
The family of intrinsic linear maps $\IL(\W,\V)$ is introduced in Section~\ref{sect:intrdiff}.
 \begin{teo}[Upper blow-up]\label{theorem:ubuIntro}
	Let $(\W,\V)$ be a couple of complementary subgroups of $\G$. Let $m$ and $M$ be the topological and the Hausdorff dimensions of $\W$, respectively. We consider an open set $A \subset \W$ and $\phi:A \to \V$. We also assume that $\phi$ is intrinsically differentiable at any point of $A$ and that $d\phi:A \to \IL(\W,\V)$ is  continuous.
	We set $\U_{w}=\mathrm{graph}(d\phi_{w})$ for every $w \in A$. Let 
	$\Phi:A\to\G$ be the graph map of $\phi$ and let us introduce the following measure
	\begin{equation}
		\mu(B)= \int_{\Phi^{-1}(B)} J\Phi(w)\  d \mathcal{H}_{|\cdot|}^m (w)
	\end{equation}
	for every Borel set $B \subset \mathbb{G}$. 
	Setting $\Sigma=\Phi(A)$, for every $x=\Phi(\zeta) \in \Sigma$ we have 
\begin{equation}\label{eq:introFedDensEq}
	\cs^M( \mu, x)= \ \beta_{d}( \U_{\zeta}).
\end{equation}
\end{teo}
In a few words, according to \eqref{eq:introFedDensEq}, we prove that for $x=\Phi(\zeta) \in \Sigma$, the spherical Federer density equals a geometric constant $\beta_{d}( \U_{\zeta})$,
associated with the homogeneous subgroup given by $\U_\zeta=\graph(d\phi_\zeta)$.  
The symbol $\beta_d(V)$ for a general linear subspace $V \subset \G$ is
the {\em spherical factor} (Definition~\ref{def:sphfactor}), which takes into account both the homogeneous 
distance used to construct the spherical measure and the fixed scalar product on $\G$. 
It generalizes the geometric constant appearing in the definition of the $k$-dimensional Hausdorff measure in Euclidean space, given by the Lebesgue measure of the unit Euclidean ball of $\R^k$. 

Combining Theorem~\ref{theorem:ubuIntro} and the measure-theoretic area formula 
of Theorem~\ref{thm:metarea}, we reach a general area formula
for the spherical measure of intrinsic graphs in homogeneous groups, which is our central result.

\begin{teo}[Area formula]\label{teo:areaIntro}
We consider a couple $(\W,\V)$ of complementary subgroups of $\G$. Let $m$ and $M$ be the topological and the Hausdorff dimensions of $\W$, respectively. We consider an open set $A \subset \W$ and a mapping $\phi:A \to \V$. We also assume that $\phi$ is intrinsically differentiable at any point of $A$ and that $d\phi:A \to \IL(\W,\V)$ is  continuous.
	Setting $\Sigma=\Phi(A)$, where $\Phi$ is the graph map of $\phi$, then for every Borel set $B \subset \Sigma $, we have the formula
\begin{equation}\label{intro:areageneral}
		\int_{\Phi^{-1}(B)} J\Phi(w)\  d \mathcal{H}_{|\cdot|}^{m} (w)= \int_B \beta_{d}(\T_x) \ d \mathcal{S}^{M}(x),
\end{equation}
	where $\T_x$ is the tangent subgroup to $\Sigma$ at $x$. 
\end{teo}
We only assume that the mapping $\phi:A\to\V$ is intrinsically differentiable with continuous intrinsic differential, that is a rather natural assumption.
The proof of this theorem corresponds to that of Theorem~\ref{teo:area}. 
The symbol $\T_x$ denotes the tangent subgroup to $\mathrm{graph}(\phi)$ at $x$ (Definition~\ref{TangenteP}) and $\cS^M$ denotes the $M$-dimensional spherical measure with respect to the fixed distance $d$ (Section~\ref{sect:FedDens}).

We first emphasize that the area formula \eqref{intro:areageneral} does not require that $\G$ is stratified. The proof of \eqref{intro:areageneral} only involves the intrinsic differentiability of the parametrizing map $\phi$, that corresponds to the blow-up of its intrinsic graph. A crucial step is to obtain
another {\em representation of the intrinsic Jacobian} $J\Phi(w)$.
In fact, setting $\U_w=\mathrm{graph}(d\phi_w)$ for every $w \in A$, by interpreting the graph map $w\to wd \phi(w)$ as the restricted group projection $\pi_{\U_w}^{\U_w, \V} |_{\W}$ associated with the splitting $(\U_w, \V)$, we obtain 
\begin{equation} \label{intro:tre}
J \Phi(w)=\frac{|\bV \wedge \bW|}{|\bV \wedge \bU_{w}|},
\end{equation}
where $\bV$ is an orienting unit $p$-vector of $\V$ and $\bW$, $\bU_{w}$ are orienting unit $(q-p)$-vectors of $\W$ and $\U_{w}$, respectively (Proposition~\ref{propjac}).
 The symbol $| \cdot |$ also denotes the norm on multivectors, which is associated with the fixed scalar product. 
 
To prove \eqref{intro:tre} we use Lemma~\ref{lm:MVPsubgroups}, which is a fundamental technical tool of the paper. The proof of this lemma passes through a number of nontrivial algebraic steps, involving the change of variables associated with the restriction of group projections. We may expect that this technical result can have its own independent interest in further developments.
 
Formula \eqref{intro:tre} is crucial to establish the continuity of $ A \ni w \to J\Phi(w)$, which is proved in Proposition~\ref{cont-jac}. Let us emphasize that such continuity does not seem achievable by elementary tools, since a representation of the intrinsic Jacobian by suitable partial derivatives is an open question in our general framework.

To make the area formula of Theorem~\ref{teo:areaIntro} more manageable for applications, we point out those cases where the area formula \eqref{intro:areageneral} takes a simpler form. We consider special classes of homogeneous distances, which are invariant under suitable families of symmetries, according to \cite[Definition 1.2]{Mag22RS}. If $\cF$ is a nonempty family of homogeneous subspaces, a homogeneous distance $d$ on a homogeneous group $\G$ is called \textit{rotationally symmetric with respect to $\cF$} if the spherical factor $\beta_d(\cdot)$ is a constant function on $\cF$. The constant value of $\beta_d(\cdot)$ is denoted by $\omega_d(\cF)$. An immediate consequence of these notions is the following ``more manageable'' area formula.

\begin{teo}\label{teo:areasymm}
	Let $(\W,\V)$ be a couple of complementary subgroups of $\G$ and denote by $m$ and $M$ the topological and the Hausdorff dimensions of $\W$, respectively. We consider an open set $A \subset \W$ and a mapping $\phi:A \to \V$. We assume that $\phi$ is intrinsically differentiable at any point of $A$ and that $d\phi:A \to \IL(\W,\V)$ is  continuous. We set $\Sigma=\Phi(A)$ and suppose that $d$ is rotationally symmetric with respect to 
	$$
	\mathcal{F}_{\V}= \{ \W' \subset \G : \W' \mathrm{ \ homogeneous \ subgroup \ complementary \ to \ } \V \}.
	$$ 
	Thus, setting $\mathcal{S}_d^M=\omega_d(\mathcal{F}_{\V}) \mathcal{S}^M$, for every Borel set $B \subset \Sigma $ we have
	\begin{equation} 
		\mathcal{S}_d^M \llcorner \Sigma(B)=\int_{\Phi^{-1}(B)} J\Phi(w)\  d \mathcal{H}_{|\cdot|}^{m} (w).
	\end{equation}
\end{teo}
The proof of the previous theorem immediately follows from Theorem \ref{teo:areaIntro}, once we know
that the tangent subgroup $\T_x$ belongs to $\cF_\V$ for every $x\in\Sigma$, which is a consequence
of Theorem~\ref{characterizationP}. 
Finding the symmetry conditions which give the hypotheses of Theorem~\ref{teo:areasymm} deserves a separate study, see \cite{Mag31} and
\cite{Mag22RS}. The latter paper includes some classes of higher codimensional, smooth submanifolds. A simple application of Theorem~\ref{teo:areasymm} is also provided at the end of the paper. 

We consider two special cases of Theorem~\ref{teo:areaIntro}. The first one concerns the intrinsic graphs of uniformly intrinsically differentiable maps. From Proposition~\ref{uidimpliescont}, this class of maps is continuously intrinsically differentiable. Then Theorem~\ref{teo:areaIntro} immediately gives Theorem~\ref{teo:areauid}. 
We stress that the family of graphs for which Theorem~\ref{teo:areaIntro} holds might be strictly larger than those of Theorem~\ref{teo:areauid}. In fact, in our general framework, it is an interesting open question to establish whether continuously intrinsically differentiable maps are uniformly intrinsically differentiable, see \cite[Corollary 4.7 and Remark 4.8]{ADDDLD24}. 

The second special case of Theorem~\ref{teo:areaIntro} refers to $(\G,\M)$-regular sets of $\G$, introduced in \cite{Mag5,Mag14} (Definition~\ref{def:GMregular}), where $\G$ and $\M$ denote two stratified Lie groups, each equipped with a homogeneous distance. This notion of regular set was also considered in \cite[Section~2.5]{JNGV22} under the terminology  ``{\em submanifold of class $C^1_H$}'' or ``{\em $C^1_H(\G;\M)$-submanifold''}.

Roughly speaking, $\Sigma \subset \G$ is a $(\G,\M)$-regular set of $\G$ if it can be locally seen as the level set of a {\em suitable} continuously differentiable map $f$ from an open subset of $\G$ to the other stratified group $\M$. Precisely, the differentiability 
is meant with respect to the homogeneous structure of the group and the corresponding differential $Df(x)$, $x\in\Sigma$, must be 
an {\em h-epimorphism} (\cite[Definition~2.5]{Mag14}), namely it is surjective and there exists a homogeneous subgroup $\V$ of $\G$ complementary to $\ker(Df(x))$. 
Under this condition on the defining map, a general implicit function theorem holds for differentiable mappings between stratified groups, \cite[Theorem 1.4]{Mag14}, see Section~\ref{secf:difimpl}.
 
Now we wish to consider Theorem~\ref{teo:areaIntro} in the case where $\G$ and $\M$ are stratified groups and the intrinsic graph $\Sigma$ is a $(\G,\M)$-regular set of $\G$. Let $Q$ and $P$ denote the Hausdorff dimensions of $\G$ and $\M$, respectively. We also indicate the topological dimensions of $\G$ and $\M$ by the integers $q$ and $p$, respectively. The symbols $J_Hf$ and $J_{\V}f$ denote the Jacobians of $Df$ and $Df|_{\V}$, respectively (Definition~\ref{def:jacf}). 
The assumptions of Theorem~\ref{teo:areaIntro} require that the implicit map of \cite[Theorem 1.4]{Mag14} is intrinsically differentiable with continuous
intrinsic differential.
By virtue of \cite[Theorem~4.3.7]{CorniPhD}, see Theorem~\ref{intdiffpar}, the implicit map of the implicit function theorem is also uniformly intrinsically differentiable, hence its intrinsic differential is continuous (Proposition~\ref{uidimpliescont}).
As a consequence of Theorem~\ref{teo:areaIntro},
we are led to a special form of the area formula for the class of $(\G,\M)$-regular sets of $\G$,
see Theorem~\ref{teo:arealevelset}.

In this theorem the Jacobian of the parametrization can be also related to the locally defining map $f$ by an intriguing ``algebraic representation'' 
\beq\label{intro:duejac}
J\Phi(w) =  |\bV \wedge \bW| \ \frac{ J_Hf(\Phi(w))}{J_{\V}f(\Phi(w))}.
\eeq
The previous formula provides another way to compute the Jacobian and when combined with Theorem~\ref{teo:areaIntro}
leads us to the following theorem.

\begin{teo}\label{teo:arealevelset}
	Let $\G$ and $\M$ be stratified groups and let $\Omega \subset \G$ be an open set, with $f \in C^1_h(\Omega, \M)$. Consider 
	$\Sigma=f^{-1}(0)$ and assume that there exist an open set $\Omega' \subset \Omega$ and a homogeneous subgroup $\V \subset \G$ of topological dimension $p$ such that
	$J_\V f(y)> 0$
	for any $y \in \Sigma \cap \Omega'$. Let $\W \subset \G$ be a homogeneous subgroup complementary to $\V$ and consider the unique map $\phi:A\to\V$, whose graph mapping 
	$\Phi:A\to\G$ satisfies $\Sigma \cap \Omega' = \Phi(A)$, where $A\subset\W$ is an open set. If $\bV$ is an orienting unit $p$-vector of $\V$ and $\bW$ is an orienting unit $(q-p)$-vector of $\W$, then we have 
	\begin{equation}\label{formula-f}
		\int_B \beta_{d}(\T_x) \ d \mathcal{S}^{Q-P}(x)= | \bV \wedge \bW| \int_{\Phi^{-1}(B)} \frac{J_Hf(\Phi(w))}{J_{\V}f(\Phi(w))} \  d \mathcal{H}_{|\cdot|}^{q-p} (w)
	\end{equation}
	for every Borel set $B \subset \Sigma \cap \Omega'$,
	where $\T_x$ is the tangent subgroup to $\Sigma$ at $x$. 
\end{teo}
The proof of \eqref{intro:duejac} can be obtained using some arguments of multilinear algebra, which are similar to those of \cite[Theorem~3.2]{CorMag22} and \cite[Theorem~5.4]{Mag12A}. 
However, in our more general setting, with intrinsic graphs of arbitrary codimension, it is rather interesting that the analogous ``algebraic arguments'' work 
using the rows of the differential $Df$ in place of the horizontal gradients of the components of $f$. We also point out that \eqref{formula-f} precisely extends \cite[Theorem 1.1]{CorMag22} from Heisenberg groups to general homogeneous groups. The proof of Theorem~\ref{teo:arealevelset} is given in Section~\ref{sect:arealevels}.

We have already noticed that Theorem~\ref{teo:areauid} applies to a class of intrinsic graphs that might be smaller than that of Theorem~\ref{teo:areaIntro}.
The $(\G,\M)$-regular sets of $\G$ in Theorem~\ref{teo:arealevelset} are uniformly intrinsically differentiable, \cite[Theorem~4.3.7]{CorniPhD},
hence they satisfy the assumptions of Theorem~\ref{teo:areauid}.
On the other hand, this class of regular sets may constitute an even smaller class than that of Theorem~\ref{teo:areauid}.
In fact, to determine whether the graph of a uniformly intrinsically differentiable
map is a $(\G,\M)$-regular set remains an intriguing open question.
The problem is related to the validity
of a suitable Whitney's extension theorem.
Indeed, for $(\G,\R^k)$-regular sets of $\G$, where this theorem is available, 
a positive answer holds, \cite[Theorem 4.1]{DiD21}.

Theorem~\ref{teo:arealevelset} can be compared with \cite[Theorem~1.1]{JNGV22}, the latter proving an area formula for the measure $\psi^d$ of \cite{JNGV22} equal to either the Hausdorff measure or the spherical measure. When $\psi^d$ is the spherical measure, one may notice that the {\em area factor} introduced in \cite{JNGV22} is proportional to the ratio between the intrinsic Jacobian $J\Phi$ in \eqref{intro:uno} and the spherical factor $\beta_d$ of Definition~\ref{def:sphfactor}. 
Another version of the area formula is proved in \cite[Theorem~1.3]{AntMer22} for a.e.\ intrinsically differentiable intrinsic-Lipschitz graphs, where the integration measure is the centered Hausdorff measure. Here the area factor is naturally replaced by the centered area factor, defined by the centered Hausdorff measure.

In Section~\ref{sect:intder}, we analyze a subclass of intrinsic regular graphs for which the spherical 
measure can be explicitly expressed in terms of suitable intrinsic partial derivatives of the parametrizing map $\phi$ (Definition \ref{intrinsicder}).
We consider maps acting between $\W$ and $\V$, where $(\W,\V)$ is a couple of complementary subgroups of $\G$ such that $\W$ and $\V$ are orthogonal, and $\V$ is a horizontal subgroup. We assume that the map is continuously intrinsically differentiable, then we combine Theorem~\ref{teo:areaIntro} 
and some results from \cite{ADDDLD24} to derive the most explicit formula
to compute the spherical measure of intrinsically regular sets. 
We can explicitly relate
the spherical measure of the intrinsic graph 
to its intrinsic partial derivatives,
according to Theorem~\ref{areaintder}.
Precisely, the partial derivatives are computed along the so-called ``projected vector fields'', see Definition~\eqref{def:projvf}. In the proof of Theorem~\ref{areaintder} the key point is to carry out a delicate computation of the intrinsic Jacobian \eqref{intro:uno} of the parametrizing map, using the matrix representation of the intrinsic differential with respect to suitable coordinates, see Theorem~\ref{teo:intJacobian}. 

The Jacobian in Theorem~\ref{areaintder}  recovers 
the previous ones in the literature for one codimensional intrinsic graphs
\cite{ASCV06, CMPSC14, DiD21, ADDDLD24}
and for low codimensional intrinsic graphs
of Heisenberg groups, \cite{Cor21,CorMag22,Vit22}.
Somehow, Theorem~\ref{areaintder} brings us back to the central point of this work, which is an effective area formula for the spherical measure with an explicit notion of Jacobian. 
It still remains rather fascinating that this area formula also holds for those sets that, from a Euclidean viewpoint, can even exhibit a fractal nature, while the classical Euclidean tools do not apply.

\section{A short compendium on homogeneous groups}\label{sect:compendium}

A \textit{graded group} $\G$ is a connected, simply connected and nilpotent Lie group, whose Lie algebra is
graded, namely there exists a sequence of subspaces $\cV_j$ with $j\in\N$, such that $\cV_j= \{ 0 \}$ if $j>\iota$, $[\cV_i, \cV_j] \subseteq \cV_{i+j}$
for every $i,j\ge1$, $\cV_\iota \neq \{ 0 \}$ and
$ \mathrm{Lie}(\G)=\cV_1 \oplus \dots \oplus \cV_\iota$, 
where 
$$[\cV_i,\cV_j]=\text{span} \{ [X,Y]  :  X \in \cV_i, \ Y \in \cV_j\}.$$ 
The positive integer $\iota$ is called the {\em step} of $\G$. 
In the special case where $[\cV_1, \cV_i] = \cV_{i+1}$ for every $i=1, \dots, \iota-1$, we call $\G$ a \textit{stratified group}.

Since $\G$ is connected, simply connected and nilpotent, the exponential map $\mathrm{exp}: \mathrm{Lie}(\G) \to \G$ is a global diffeomorphism that permits us to identify in a standard way $\G$ with $\mathrm{Lie}(\G)$.
As a consequence, we may think of a graded group $\G$ as a graded vector space $V_1 \oplus V_2 \oplus \dots \oplus V_\iota$ endowed with both a Lie group and a Lie algebra structure. 

For the sequel, it is convenient to introduce the following integers:
\[
h_0=0,\quad n_i=\dim(V_i) \quad \text{and}\quad h_s=\sum_{j=1}^{s}n_j,
\]
for every $i,s=1,\ldots,\iota$. The integer $q$ indicates the linear dimension of $\G$. 
\begin{deff}[Graded basis]
We say that the basis $(e_1,\ldots,e_q)$ of a graded Lie group $\G$ is {\em graded} if
the ordered families $ 	(e_{h_{i-1}+1},\ldots,e_{h_i}) $ constitute a basis of $V_i$ for every $i=1,\ldots,\iota$.
\end{deff}

We introduce the Lie group operation on $\G$ through the well known Baker--Campbell--Hausdorff formula, abbreviated as the {\em BCH formula}, see for instance \cite[Section~2.15]{Var84Lie}. We may state the BCH formula with respect to a fixed graded basis $(v_1, \ldots,v_q)$ of $\G$. Then for every $v,w\in\G$ the formula reads as
	\beq \label{eq:BCH}
	vw=v+w+ \sum_{s=2}^\iota \sum_{j=h_{s-1}+1}^{h_s} Q_j\big(\bar{\cp}_{s-1}(v),\bar{\cp}_{s-1}(w)\big)v_j,
	\eeq 
where we have defined the projections
\begin{equation}\label{eq: projections}
		\bar{\cp}_j:\G\to V_1\oplus V_2\oplus\cdots\oplus V_j, \quad \bar{\cp}_j(z)=\sum_{i=1}^j z_j, \quad z=\sum_{i=1}^\iota z_j, \ z_j\in V_j
	\end{equation}
and $Q_j$ are polynomial functions on $\G$. The BCH formula could be also given with a coordinate-free representation, see \cite[Lemma 2.15.3]{Var84Lie}.
We can easily introduce natural dilations on $\G$ that respect its group structure
\begin{equation}\label{dilation} \delta_t(x)=\sum_{i=1}^\iota t^i x_i \quad \mathrm{if} \quad  x= \sum_{i=1}^\iota x_i, \ x_i \in V_i \quad \text{and} \quad t>0.
\end{equation}
In the paper we always deal with this class of Lie groups, if not otherwise specified.

We can equip a graded group $\G$ with a {\em homogeneous distance}, i.e.\ a distance $d$ on $\G$ such that for every $x,y,z \in \G$ and $t>0$, it satisfies the conditions $d(zx,zy)=d(x,y)$ and $d(\delta_tx, \delta_ty)=t d(x,y)$. We also introduce the {\em homogeneous norm} $\| x \|=d(x,0)$ for every $x \in \G$.
A {\em homogeneous group} is a graded group equipped with the family of dilations $\set{\delta_r:r>0}$ and a homogeneous distance.
Dilations also define the special class of {\em homogeneous subgroups},
namely Lie subgroups which are closed with respect to dilations. 
They are also subspaces with respect to the linear structure of $\G$.

It is well known that the Hausdorff dimension of $\G$ with respect to $d$ is 
\[
Q=\sum_{k=1}^{\iota} k \ \mathrm{dim}(V_k)
\]
and it does not depend on the choice of the homogeneous distance, since they are all equivalent to each other. 
For $x \in \G$ and $r>0$, we define the closed ball 
$$
\B(x,r)= \{ y \in \G: d(x,y) \leq r \}
$$
and the open ball 
$$B(x,r)=\{y \in \G : d(x,y)<r \}.$$
We will also use the distance function with respect to a set $A \subset \G$: for $x \in \G$, we set 
\[
\mathrm{dist}(x, A)= \text{inf} \set{ d(x,y)  : y \in A }.
\]
Throughout the paper, a scalar product $ \langle \cdot, \cdot \rangle$ on $\G$ is fixed and we denote by $|\cdot|$ its associated norm. Taking into account the linear structure of $\G$, we have a canonical isomorphism between $\G$ and $T_0\G$. Therefore the scalar product $\langle\cdot,\cdot\rangle$ automatically extends to a left invariant Riemannian metric $g$ on $\G$. The norm generated by the inner product on the tangent space $T_x\G$ is denoted by $|\cdot|_g$, with $x \in \G$. 
For every $k \in \N$ we consider the space $\Lambda_k \G$ of $k$-vectors. The fixed scalar product $\langle \cdot, \cdot \rangle$ naturally extends to a scalar product on $\Lambda_k \G$. We have then defined a Hilbert space structure on $\Lambda_k\G$, where the norm is still denoted by $| \cdot |$.

\section{Intrinsic regularity in homogeneous groups}\label{sect:NotFactHom}

In this section, we present some definitions and results related to factorizations of homogeneous groups, intrinsic graphs, intrinsic Lipschitz mappings, intrinsically linear maps and intrinsic differentiability in homogeneous groups.
For this topic, we refer for instance to \cite{FMS14,FranchiSerapioni2016IntrLip,SerraCassano2016,CorniPhD}. Although in these works the theories are presented in stratified groups, it is not difficult to notice that the results we use do not rely on the Lie bracket generating condition.

In the sequel $\G$ is assumed to be a homogeneous group. 
Two homogeneous subgroups $\W$ and $\V$ are said to be \textit{complementary subgroups of $\G$} if
\begin{equation}\label{eq:compl}
\W \cap \V = \{0 \}\quad \text{and} \quad \G= \W   \V.
\end{equation}
Such conditions can be rephrased as follows: for every $ x \in \G$ there exists a unique couple 
\[ 
(w,v) \in \W\times \V\quad \text{such that}\quad   x= w   v.
\]
As a consequence, the ``group projections'' 
\beq\label{eq:proj}
\pi_{\W}: \G \to \W, \ \pi_{\W}(wv)=w, \ \pi_{\V}: \G \to \V, \ \pi_{\V}(wv)=v
\eeq
are well defined for every $w\in\W$ and $v\in\V$.
Notice that the order in the choice of $\W$ and $\V$ matters. If we consider the reverse order $\G=\V\W$, then the noncommutative group operation yields two different group projections. 
For this reason we use the notation $(\W,\V)$ to denote the (ordered) {\em couple of complementary subgroups} $\W$ and $\V$ such that the group projections are defined in \eqref{eq:proj}.
To emphasize the dependence on the order of the factorization, 
we also write
$\pi^{\W,\V}_\W=\pi_\W$ and $\pi^{\W,\V}_\V=\pi_\V$. 

Throughout the section, a couple $(\W,\V)$ of complementary subgroups of $\G$ is understood, if not otherwise stated.

\begin{deff}\label{d:projWVM}
Let $\W$, $\elle$, $\V$ be homogeneous subgroups of $\G$. If $(\W,\V)$ and $(\elle,\V)$ are couples of complementary subgroups, then we define the following restrictions of the group projections
\[
\pi_{\W,\elle}^{\W,\V}=\pi^{\W,\V}_\W|_{\elle}:\elle \to \W\quad\text{and}\quad 
\pi_{\elle,\W}^{\elle,\V}=\pi^{\elle,\V}_\elle|_{\W}:\W\to \elle.
\]
\end{deff}
From the uniqueness of factorizations, we can verify that both the restrictions $\pi_{\W,\elle}^{\W,\V}$ and $\pi_{\elle, \W}^{\elle,\V}$ are invertible and we have the formula
\begin{equation}\label{eq:invproj}
\pi_{\W,\elle}^{\W,\V}=(\pi_{\elle,\W}^{\elle,\V})^{-1}.
\end{equation}
One can easily realize that  
\begin{equation}\label{defczero}
c_0=c_0(\W,\V)=\inf \{ \| wv\| : w \in \W, \ v \in \V, \ \|w \|+ \|v \|=1 \}>0.
\end{equation}
By the homogeneity of the distance and the triangle inequality, we have
\begin{equation}\label{eqczero}
c_0 ( \| w \|+ \|v \|) \leq \|wv\| \leq \|w \|+ \|v \|
\end{equation}
for all $w \in \W$, $v \in \V$.
As a direct consequence of \eqref{eqczero}, for all $x\in\G$ we also observe that
\beq\label{dis:projest}
c_0 \| \pi_{\V}(x)  \| \leq \mathrm{dist}(x,\W) \leq \| \pi_{\V}(x)  \|.
\eeq

\begin{deff}[Intrinsic graph]
For $A \subset\W$, we define the \textit{intrinsic graph} of $\phi: A \to \mathbb{V}$ as the set
$$
 \mathrm{graph}(\phi)=\{w\phi(w):  w \in A \}.
$$
The \emph{graph map} $\Phi: A \to \G$ of $\phi$ is defined by $ \Phi(w):= w \phi(w)$ for all $w \in A$.
\end{deff}

Translating intrinsic graphs requires some preliminary notions. 

\begin{deff}[Translations with respect to a factorization]\label{def-translation}
For each $x \in \G$ we define the map $\sigma_x:\W \to \W$ as
\beq\label{df:sigma_x}
\sigma_x(\eta)=\pi_{\W}(x \eta)
\eeq
for every $\eta\in\W$. For $A \subset \W$ and $\phi:A  \to \V$, we consider $A_x= \sigma_x(A)$ to be the 
{\em translated set by $x$} and introduce the map $\phi_{x}: A_x \to\V$ as
\begin{equation}\label{eq:translatedf}
\phi_x(\eta)=(\pi_\V(x^{-1} \eta))^{-1}\phi(\pi_\W(x^{-1} \eta))=(\pi_\V(x^{-1} \eta))^{-1}\phi(\sigma_{x^{-1}}(\eta))
\end{equation}
for every $\eta\in A_x$, that is the {\em translation of $\phi$ by $x$}.
\end{deff}

It is not difficult to realize that $\sigma_x:\W\to\W$ is invertible and 
\beq
(\sigma_x)^{-1}=\sigma_{x^{-1}}
\eeq
for every $x\in\G$. Indeed, for every $w \in \W$ we have
\begin{equation*}
\sigma_{x^{-1}}(\sigma_x(w))=\pi_{\W}(x^{-1}\pi_{\W}(xw))=\pi_{\W}(x^{-1}xw(\pi_{\V}(xw))^{-1})=\pi_{\W}(w(\pi_{\V}(xw))^{-1})=w.
\end{equation*}
For every $x\in\G$, $A\subset W$ and $\phi:A\to\V$, the equalities
\beq\label{eq:translgraph}
 x(\mathrm{graph}(\phi))= \mathrm{graph}(\phi_x)= \{ \eta \phi_x(\eta) : \eta \in A_x \}
\eeq 
are a direct consequence of Definition~\ref{def-translation}, where $A_x$ denotes the translated set by $x$.

\subsection{Intrinsic Lipschitz maps}

The present section collects some basic facts on intrinsic Lipschitz maps in a homogeneous group $\G$.
Notice that there are different equivalent definitions of intrinsic Lipschitz functions, as discussed below.
We recall that a couple of complementary subgroups $(\W,\V)$ of $\G$ is fixed also in this section.

\begin{deff}\label{intlip}
For some $L\ge 0$, we say that $\phi:A \to \V$, $A \subset \W$, is \textit{intrinsic $L$-Lipschitz} if
for every $w ,w' \in A$ we have
$$
 \| \pi_{\V}( \Phi(w')^{-1} \Phi(w) ) \| \leq L \| \pi_{\W}( \Phi(w')^{-1} \Phi(w) ) \|.
$$
We denote by $\mathrm{Lip}(\phi)\ge0$ the infimum among the positive constants $L$ such that $\phi$ is intrinsic $L$-Lipschitz.
We say that $\phi$ is \textit{intrinsic Lipschitz} if there exists $L>0$ such that $\phi$ is intrinsic $L$-Lipschitz.
\end{deff}

It is easy to notice that \cite[Proposition 3.3]{FranchiSerapioni2016IntrLip} also holds in homogeneous groups, as stated in the next proposition.

\begin{prop}\label{charintlip}
For $L\ge0$ and $A \subset \W$ and $\phi:A \to \V$, $\phi$ is intrinsic $L$-Lipschitz
if and only if $ \| \phi_{x^{-1}}(w) \| \leq L \| w \| $ for every $x \in \mathrm{graph}(\phi)$ and $w \in A_{x^{-1}}$.
\end{prop}

If $\W\subset\G$ is a normal subgroup and $A\subset\W$, 
then the intrinsic $L$-Lipschitz continuity of $\phi:A\to\V$ can be simply rephrased as
\begin{equation}\label{intlipWnormal}
		\| \phi(w')^{-1} \phi(w) \| \leq L \| \phi(w')^{-1} w'^{-1} w \phi(w') \|
\end{equation}
	for every $w,w' \in A$.
\subsubsection{Intrinsic Lipschitz maps by cones} 
It is well known that intrinsic Lipschitz maps can be equivalently defined by cones constructed by the homogeneous distance. Actually one can introduce a more general notion of 
intrinsic Lipschitz $\H$-graph, that only refers to a subset of $\G$ and to a homogeneous subgroup $\H\subset\G$,  for instance \cite[Definition~11]{FranchiSerapioni2016IntrLip}.  

To introduce this definition we first define cones in homogeneous groups.
Let $\H\subset \G$ be any homogeneous subgroup, let $x \in \G$ and $\alpha \in (0,1)$. 
The \textit{cone} with \textit{axis} $\H$, \textit{vertex} $x$ and \textit{opening} $\alpha$ is $X(x,\H, \alpha)=x  X(0,\H, \alpha)$ where 
\[
	X(0,\H, \alpha)= \set{y \in \G  : \text{dist}(y, \H) \leq \alpha \| y \|}.
\]
We say that $S\subset \G$ is an {\em intrinsic Lipschitz $\H$-graph} if 
\[
S\cap X(p,\H,\alpha)={p}\quad \text{for every $p\in S$.}
\]
Taking into account \cite[Proposition 3.1]{FranchiSerapioni2016IntrLip} and \cite[Proposition~3.3]{FranchiSerapioni2016IntrLip}, one can verify that $\phi:A\to\V$ is intrinsic Lipschitz if and only if
its intrinsic graph is an intrinsic Lipschitz $\V$-graph, namely, there exists $\alpha \in (0,1)$ such that
\begin{equation}\label{eq:graphLip}
\mathrm{graph}(\phi) \cap X(p, \V, \alpha)= \{p\} \ \ \ \ \mathrm{for \ every \ } x \in \mathrm{graph}(\phi).
\end{equation}

\subsection{Intrinsic differentiability}\label{sect:intrdiff}
A couple of complementary subgroups $(\W,\V)$ of the homogeneous group $\G$ is fixed also in this section. We first introduce some basic facts about intrinsically linear mappings.

We say that the function $L: \W \to\V$ is \textit{intrinsically linear} if $\text{graph(}L\text{)}$ is a homogeneous subgroup of $\G$. It is straightforward to notice that any intrinsically linear function $L: \W \to \V $ is also homogeneous i.e.\ $L(\delta_t(w))=\delta_t(L(w))$ for every $w \in \W$, $t>0$.
We denote by $\IL(\W,\V)$ the family of intrinsically linear maps from $\W$ to $\V$.

For $L, T \in \IL\pa{\W,\V}$, we define the distance
\begin{equation}\label{eq:distwv}
\ccS_{\W,\V}(L,T)= \sup \{ d(L(w),T(w)) : w \in \W, \ \| w \|=1 \}.
\end{equation}
We assume throughout that $\IL(\W,\V)$ is topologized by this distance.

The following proposition corresponds to \cite[Proposition 3.1.5]{FMS14}.

\begin{prop}\label{P7}
The following statements hold.
\begin{itemize}
\item[(i)] If $L: \W \to \V$ is intrinsically linear, then $(\mathrm{graph}(L),\V)$ represents a couple of complementary subgroups of $\G$.
\item[(ii)] If $(\mathbb{H},\V)$ is a couple of complementary subgroups of $\G$, then there exists a unique intrinsically linear function $L: \W \to \V$ such that $\mathbb{H}= \mathrm{graph}(L)$.
\end{itemize}
\end{prop}

The following proposition corresponds to \cite[Proposition 3.1.6]{FMS14}
\begin{prop}
\label{intlinintlip}
If $L: \W \to \V$ is an intrinsically linear map, then $L$ is intrinsic Lipschitz.
\end{prop}

\begin{deff}[Intrinsic differentiability]\label{d:intrinsicDiff}
Let $A \subset \W$ be an open set, let $\phi: A  \to \V$ and $\bar{w} \in A$. Defining $x=\bar{w} \phi(\bar{w})$, we say that $\phi$ is \textit{intrinsically differentiable} at $\bar{w}$ if there exists an intrinsically linear map $L: \W \to \V$ such that 
\begin{equation}\label{intdiff}
\|L(w)^{-1} \phi_{x^{-1}}(w)\| =o(\| w \|)
\end{equation} 
as $ \|w \| \to 0$ and  $w \in A_{x^{-1}}$.
If the function $L$ exists, it is unique and we call it the {\em intrinsic differential} of $\phi$ at $\bar{w}$, denoted by $d \phi_{\bar{w}}$.
\end{deff}

Taking into account the standard relationship between the differential of a mapping and the tangent space to its graph, 
following \cite[Definition~3.2.6]{FMS14} we introduce the notion of tangent subgroup.

\begin{deff}[Tangent subgroup]\label{TangenteP}
Let $A \subset \W$ be an open set and let $\phi:A  \to \V$ be a function. Let us fix $\bar{w} \in A$ and consider a point $\bar{x}= \bar{w} \phi(\bar{w}) \in \text{graph(}\phi \text{)}$. 
We say that a homogeneous subgroup $\T$ of $\G$ is the \textit{tangent subgroup} to $\mathrm{graph}(\phi)$ at $\bar{x}$, if for all $\varepsilon >0$ there exists $\delta>0$ such that we have
\begin{equation}\label{condizionetangenzaP}
		\text{graph(}\phi_{{\bar{x}}^{-1}} \text{)} \cap \{ x \in \G  : \|\pi_{\W}(x) \| < \delta \}   \subset X(0,\mathbb{T}, \varepsilon).
	\end{equation}
\end{deff}

The next theorem, proved in \cite[Theorem~3.2.8]{FMS14}, connects the tangent subgroup with the intrinsic differential. 

\begin{teo}\label{characterizationP}
Let $A \subset \W$ be an open set and let $\phi:A  \to \V$. We consider $\bar{w} \in A$ and the point $\bar{x}= \Phi(\bar{w}) \in \mathrm{graph}(\phi)$. Then the following conditions are equivalent.
\begin{itemize}
		\item[(i)] $\phi$ is intrinsically differentiable at $\bar{w} \in A$.
		\item[(ii)] There exists a set $\T$ such that
		\begin{itemize}
			\item[($ii_1$)] $\T$ is a homogeneous subgroup;
			\item[($ii_2$)] $(\T,\V)$ is a couple of complementary subgroups of $\G$;
			\item[($ii_3$)] $ \T$ is the tangent subgroup to $\mathrm{graph}(\phi)$ at $\bar{x}$.
		\end{itemize}
\end{itemize}
	Moreover, if either $(i)$ or $(ii)$ holds, then the intrinsic differential $d\phi_{\bar{w}}:\W \to \V$ is the unique intrinsically linear function such that $\mathrm{graph}(d\phi_{\bar{w}})=\T$.
\end{teo}

A stronger notion of pointwise intrinsic differentiability is given in the next definition. 

\begin{deff}[Uniform intrinsic differentiability]\label{uniformintdiff}
Let $A \subset \W$ be an open set and let $\phi: A  \to \V$ be a function. The map $\phi$ is {\em uniformly intrinsically differentiable} at a point $\bar{w} \in A$ if there exists an intrinsically linear map $L: \W \to \V$ such that
\begin{equation}\label{equid}
\lim_{r \to 0^+}\sup_{ \|\bar{w}^{-1} w' \| < r} \sup_{0< \| w \|< r}\frac{ \| L(w)^{-1}\phi_{\Phi(w')^{-1}}(w) \|}{\| w \|} =0,
\end{equation}
where $\Phi$ denotes the graph map of $\phi$, $w' \in A$ and $w \in A_{\Phi(w')^{-1}}$.
The map $\phi$ is said uniformly intrinsically differentiable on $A$ if it is uniformly intrinsically differentiable at $w$ for every $w \in A$.
\end{deff}

A version of the previous definition first appeared in \cite[Definition 3.16]{AS2009}. We have referred to the more recent version in \cite[Definition 3.3]{DiD21}.

Next, we show that the uniform intrinsic differentiability implies the continuity of the intrinsic differential. For a homogeneous subgroup $\W\subset \G$, $x \in \W$ and $r>0$, we set 
\[
B_{\W}(x,r)=B(x,r) \cap \W =\{ y \in \W : d(y,x)<r\}\quad\text{and}\quad B_{\W}^{*}(x,r)=B_{\W}(x,r)\setminus \{ x \}.
\]
The following proposition extends \cite[Proposition 3.7(iii)]{DiD21}.
\begin{prop}\label{uidimpliescont}
	Let $(\W,\V)$ be a couple of complementary subgroups of $\G$. Let $A \subset \W$ be an open set and let $\phi: A  \to \V$ be uniformly intrinsically differentiable on $A$. Then the intrinsic differential $d \phi: A \to \IL(\W,\V), \  w \to d\phi_w$ is continuous on $A$.
\end{prop}
\begin{proof}
	Let $w_0 \in A$. We want to show that 
	$$\lim_{w\to w_0} \ccS_{\W,\V}(d\phi_{w_0}, d\phi_{w}) =0.$$
	By contradiction, we assume that there exists $\varepsilon>0$ such that for every 
	$n \in \N\sm\set{0}$, there exists $u_n\in B_\W(w_0,1/n)$ and 
	$$
	\sup_{\|w\| =1} \|(d\phi_{w_0}(w))^{-1} d \phi_{u_n}(w) \| >\varepsilon.
	$$
	Then we can find $w_n \in \W$ such that $\|w_n\|=1$ and 
	\begin{equation}\label{condizione}
		\|(d\phi_{w_0}(w_n))^{-1} d \phi_{u_n}(w_n) \| > \varepsilon.
	\end{equation}
	By the homogeneity of the intrinsic differential and \eqref{condizione}, selecting $0<t <\frac{1}{n}$, we have
	\begin{equation}
		\|\delta_{\frac{1}{t}}(d\phi_{w_0}(\delta_t(w_n))^{-1} d \phi_{u_n}(\delta_t(w_n)) \| =\|\delta_{\frac{1}{t\|w_n\|}}(d\phi_{w_0}(\delta_t(w_n))^{-1} d \phi_{u_n}(\delta_t(w_n)) \|> \varepsilon.
	\end{equation}
	From the uniform intrinsic differentiability at $w_0$ we get $n_\ep\in\N\sm\set{0}$ such that 
	\beq
	\sup_{w' \in B_{\W}\big(w_0, \frac{1}{n}\big)} \sup_{w \in B_{\W}^{\star}\big(0,\frac{1}{n}\big)} \frac{\|(d\phi_{w_0}(w)^{-1} \phi_{\Phi(w')^{-1}}(w))\|}{\| w\|}<\frac\ep4
	\eeq
	for every $n\ge n_\ep$. By the triangle inequality, for $0<t<1/n_\ep$, we get 
	\begin{align*}
		\ep <&\|\delta_{\frac{1}{t\|w_{n_{\ep}}\|}}(d\phi_{w_0}(\delta_t(w_{n_{\ep}}))^{-1} \phi_{\Phi(u_{n_{\ep}})^{-1}}(\delta_t(w_{n_{\ep}}))) \delta_{\frac{1}{t\|w_{n_{\ep}}\|}}( \big(\phi_{\Phi(u_{n_{\ep}})^{-1}}(\delta_t(w_{n_{\ep}}))\big)^{-1} d \phi_{u_{n_{\ep}}}(\delta_t(w_{n_{\ep}})) \|\\
		\leq & \frac{\|(d\phi_{w_0}(\delta_t(w_{n_{\ep}}))^{-1} \phi_{\Phi(u_{n_{\ep}})^{-1}}(\delta_t(w_{n_{\ep}})))\|}{\| \delta_t(w_{n_{\ep}})\|} + \frac{\|   d \phi_{u_{n_{\ep}}}(\delta_t(w_{n_{\ep}}))^{-1}\phi_{\Phi(u_{n_{\ep}})^{-1}}(\delta_t(w_{n_{\ep}})) \|}{\| \delta_t(w_{n_{\ep}}) \|}\\
		\leq & \sup_{w' \in B_{\W}\big(w_0, \frac{1}{n_\ep}\big)} \frac{\|(d\phi_{w_0}(\delta_t(w_{n_{\ep}}))^{-1} \phi_{\Phi(w')^{-1}}(\delta_t(w_{n_{\ep}})))\|}{\| \delta_t(w_{n_{\ep}})\|} + \frac{\|   d \phi_{u_{n_{\ep}}}(\delta_t(w_{n_{\ep}}))^{-1}\phi_{\Phi(u_{n_{\ep}})^{-1}}(\delta_t(w_{n_{\ep}})) \|}{\| \delta_t(w_{n_{\ep}}) \|}
		\\
		\leq & \sup_{w' \in B_{\W}\big(w_0, \frac{1}{n_\ep}\big)} \sup_{w \in B_{\W}^{\star}\big(0,\frac{1}{n_\ep}\big)} \frac{\|(d\phi_{w_0}(w)^{-1} \phi_{\Phi(w')^{-1}}(w))\|}{\| w\|} + \frac\ep4<\frac\ep2.
	\end{align*}
	Due to the intrinsic differentiability of $\phi$ at $u_{n_\ep}$, in the previous inequalities the second addend is less than $\ep/4$, up to taking $0<t<\min{t_\ep,1/n_\ep}$ for some suitable $t_\ep>0$, hence reaching a contradiction.
\end{proof}

\section{Spherical measure and Federer density}\label{sect:FedDens}

In the present section, we introduce some measure-theoretic tools that will be used to compute the spherical 
measure of intrinsic graphs. Let $\mathcal{F} \subset \mathcal{P}(\G) $ be a nonempty family of closed subsets of a homogeneous group $\G$ equipped with a homogeneous distance $d$. 

Let $\zeta: \mathcal{F} \to [0,+\infty]$ and for $\delta>0$, $A \subset \G$ define
\begin{equation}\label{caratheodory}
	\phi_{\delta, \zeta}(A)= \inf \set{ \sum_{j=0}^{\infty} \zeta(B_j) \ : \ A \subset \bigcup_{j=0}^\infty B_j , \ \mathrm{diam}(B_j) \leq \delta, \ B_j \in \mathcal{F} }.
\end{equation}
Taking $\phi_\zeta(A)=\sup_{\delta>0}\phi_{\delta,\zeta}(A)$, we have obtained a Borel regular measure $\phi_\zeta$ over the metric space $\G$. 
We introduce the gauge
\[
\zeta_\alpha(S)=(\diam(S)/2)^{\alpha}
\]
for every $S\subset \G$. 
If $\mathcal{F}$ coincides with the family $\cF_b$ of closed balls with positive radius and we 
consider $\zeta=\zeta_\alpha|_{\cF_b}$, then we call the resulting measure $\phi_{\zeta_\alpha}$ the \textit{$\alpha$-dimensional spherical measure} and denote it by $\cS^\alpha$.

We also consider the case where $\mathcal{F}$ is the family $\cF_c$ of closed sets, $\G$ is equipped with a scalar product $\lan \cdot,\cdot\ran$ and its associated norm $|\cdot|$. Then we fix $k \in \{1, \dots, q \}$ and define the geometric constant
$$
\omega_k=\cL^k(\set{x\in\R^k: |x|_{\R^k}\le 1}),
$$
where $|\cdot|_{\R^k}$ is the Euclidean norm of $\R^k$. 
Considering $\zeta=\zeta_k|_{\cF_c}$, the associated measure
$\omega_k\phi_{\zeta_k}$ is the \textit{$k$-dimensional Hausdorff measure}, 
that we denote by $\mathcal{H}^k_{|\cdot|}$. 
Moreover, let $n \in \N$ and consider $\mathcal{F}_E$ the family of closed sets of $\R^n$ equipped with the Euclidean norm $| \cdot |_{\R^n}$. If we fix $k \in \{1, \dots, n \}$ and we consider $\zeta=\zeta_k|_{\cF_E}$, the associated measure
$\omega_k\phi_{\zeta}$ is the \textit{$k$-dimensional Euclidean Hausdorff measure} and we denote it by $\mathcal{H}^k_{E}$.

\begin{deff}[Spherical Federer density]\label{def:sphefederer}
Let $\alpha>0$, $x \in \G$ and let $\mu$ be a Borel regular measure over $\G$. We define the \emph{spherical Federer $\alpha$-density of $\mu$ at $x\in\G$} by
the formula
\[
\cs^{\alpha}(\mu,x)= \inf_{\varepsilon>0} \sup \left\{ \frac{\mu(\mathbb{B})}{r(\mathbb{B})^{\alpha}} : x \in \mathbb{B} \in \mathcal{F}_b, \ \mathrm{diam}(\mathbb{B}) < \varepsilon \right\}\in[0,+\infty],
\]
where $r(\B)$ denotes the radius of the metric ball.
\end{deff}

The spherical Federer density was first introduced in \cite{Mag30} to establish the corresponding {\em measure-theoretic area formula}, see \cite{Mag30}, and \cite{LecMag22} for more details.

\begin{The}[{\cite[Theorem~5.7]{LecMag22}}]\label{thm:metarea}
Let $\mu$ be an outer measure over a diametrically regular metric space $X$ and let $\alpha>0$. 
We choose a Borel set $A\subset X$ and assume the validity of the following conditions.
	\begin{enumerate}
		\item 
		$\mu$ is both a regular measure and a Borel measure.
		\item
		$(\cF_b)_{\mu,\zeta_{\alpha}}$ covers $A$ finely.
		\item
		$A$ has a countable covering whose elements are open and have $\mu$-finite measure.
		\item
		The subset $\lbrace x\in A: \cs^\alpha(\mu,x)=0\rbrace$ is
		$\sigma$-finite with respect to $\cS^\alpha$.
		\item
		We have the absolute continuity $\mu\res A<<\cS^\alpha\res A$.
	\end{enumerate}
	Then $\cs^\alpha(\mu,\cdot):A\to[0,+\infty]$ is Borel and for every Borel set $B\subset A$ we have 
	\begin{equation}\label{eq:metarea}
		\mu(B)=\int_B\cs^\alpha(\mu,x)\,d\cS^\alpha(x).
	\end{equation}
\end{The}

We introduce some terminology needed to apply Theorem~\ref{thm:metarea}.
A {\em diametrically regular metric space} $(X,d)$ has the property that for each $x\in X$ there exist $R_x,\delta_x>0$ such that the function $r\to \diam(B(y,r))$ is continuous on $(0,\delta_x)$ for every $y\in B(x,R_x)$.
All homogeneous groups are obviously diametrically regular, since $\diam(\B(x,r))=2r$ for every $r>0$ and $x\in\G$.

We say that a class of subsets $\cS\subset \cP(X)$ covers $A\subset X$ {\em finely} if the following condition holds.
For every $x\in A$ and $\ep>0$ we have some $S\in\cS$ such that $x\in S$ and $\diam(S)<\ep$.
From  \cite[Definition~3]{LecMag22}, it is easy to notice that $(\cF_b)_{\mu,\zeta_{\alpha}}=\cF_b$, hence this class of closed sets obviously covers finely every subset of $\G$.

\section{Differentiability and implicit function theorem}\label{secf:difimpl}

In this section we recall some known facts about differentiability, implicit function theorem and the related notion of $(\G,\M)$-regular set. In Theorem~\ref{corintdiffparam} we provide a stronger version of the implicit function theorem \cite[Theorem~1.4]{Mag14}.
Another version of this theorem appears in \cite[Lemma~2.10]{JNGV22}, whose proof relies on the strict Pansu differentiability of continuously 
Pansu differentiable mappings, which in turn follows directly from the mean value estimate \cite[Theorem~1.2]{Mag14}.

The symbols $\G$ and $\M$ denote two homogeneous groups
equipped with the gradings 
$$
\G=V_1 \oplus V_2 \oplus \dots \oplus V_\iota \quad \mathrm{  and  } \quad \M= W_1 \oplus W_2 \oplus \dots \oplus W_\upsilon.
$$ 
The linear dimensions of $\G$ and $\M$ are $q$ and $p$, respectively. 
We assume that $\G$ is equipped with a homogeneous distance $d$ and $\M$ with a homogeneous distance $\rho$. 
The Hausdorff dimensions of $\G$ and $\M$ with respect to their homogeneous distance are denoted by
$Q$ and $P$, respectively.
Two scalar products on $\G$ and $\M$ are also understood. 
The homogeneous norm on $\G$ induced by $d$ is denoted by $\|\cdot\|$.

A Lie group homomorphism $L:\G \to \M$ is an {\em h-homomorphism} if for any $x \in \G$ and $t>0$ we have $L(\delta_t(x))=\delta_t(L(x))$, where we have denoted by  the same symbol $\delta_t$ both the dilations of $\G$ and $\M$. 
The family of all h-homomorphisms from $\G$ to $\M$ is denoted by $\mathcal{L}_h(\G, \M)$. 
For $L,T \in \mathcal{L}_h(\G, \M)$, we define the distance
\begin{equation}\label{eq:nu}
\nu(L,T)= \sup \{ \rho(L(y),T(y)) : y \in \G, \ \| y \|\leq 1 \}.
\end{equation}
We assume that $\mathcal{L}_h(\G, \M)$ is equipped with the topology induced by this distance.

\begin{deff}[Differential with respect to the group structure]
Let $\Omega\subset\G$ be an open set and let $x\in\Omega$. We say that $f : \Omega \to \M$ is {\em h-differentiable} at $x \in \Omega$, or simply {\em differentiable} at $x$, if there exists an h-homomorphism $L:\G \to \M$ such that
$$
\rho(L(x^{-1} y),f(x)^{-1} f(y))= o ( d(x,y)) \quad \text{as} \quad  y\to x,
$$
then it is unique and it is called the {\em h-differential of $f$ at $x$}, or the {\em Pansu differential}, or simply the {\em differential of $f$ at $x$}. We denote this Lie group homomorphism by $Df(x)$. 
\end{deff}

We say that $f:\Omega\to\M$ is {\em continuously differentiable in $\Omega$},
if $Df: \Omega \to \mathcal{L}_h(\G, \M)$ is continuous. 
The family of all these mappings is denoted by
$C^1_h(\Omega, \M)$.

\begin{deff}[Jacobian]\label{def:jacf}
	Let $\Omega \subset \G$ be an open set and let $f \in C^1_h(\Omega, \M)$. Let $x \in \Omega$ and let $Df(x)$ be the differential of $f$ at $x$. Since $Df(x): \G \to \M$ is also linear, we can consider the following canonical extension of $Df(x)$ to $p$-vectors of $\G$
	$$ 
	\Lambda_pDf(x): \Lambda_p\G \to \Lambda_p\M, \ \Lambda_pDf(x)(v_1 \wedge \dots \wedge v_p)=Df(x)(v_1) \wedge \dots \wedge Df(x)(v_p),
	$$
	for $v_1,\ldots,v_p\in\G$.
	The fixed scalar products on $\G$ and $\M$ naturally provide scalar products on $\Lambda_p\G$ and $\Lambda_p\M$, respectively. The {\em Jacobian} of $L$ is defined as the operator norm $\|\Lambda_pDf(x)\|$ of $\Lambda_pDf(x)$ and it is denoted by $J_Hf(x)$. 
	In a similar way, if $\V \subset \G$ is a homogeneous subgroup we set $J_{\V}f(x)= \|\Lambda_p(Df(x)|_{\V}) \|$.
\end{deff}

\begin{deff}[{\cite[Definition 2.5]{Mag14}}]
We say that an h-homomorphism $L: \G \to \M$ is an \textit{h-epimorphism} 
if it has a right inverse that is also an h-homomorphism. 
\end{deff}

The next proposition extends \cite[Proposition 7.10]{Mag14} with an additional characterization of h-epimorphisms. 

\begin{prop}\label{char-hepi}
	Let $L: \G \to \M$ be an h-homomorphism and let $\K$ be its kernel. Then the following conditions are equivalent:
	\begin{itemize}
		\item[(i)] $L$ is an h-epimorphism,
		\item[(ii)]  $L$ is surjective and there exists a homogeneous subgroup $\V \subset \G$ which is complementary to $\K$,
		\item[(iii)] there exists a homogeneous subgroup $\V \subset \G$ such that $L|_{\V}:\V \to \M$ is an h-isomorphism.
	\end{itemize} 
\end{prop}
\begin{proof}
	The equivalence of (i) and (ii) has been proved in \cite[Proposition 7.10]{Mag14}, which ensures also that (i) implies (iii).
	Let us assume that (iii) is true. It is immediate to notice that the inverse mapping $T=(L|_{\V})^{-1}$ is a right inverse of $L$ and therefore (i) is verified.
\end{proof}

\begin{teo}[{{\cite[Theorem 1.4]{Mag14}}}]
	\label{IFT}
	Let $\G$ and $\M$ be stratified groups, let $\Omega \subset \G$ be open and consider $f \in C^1_h(\Omega, \M)$. Let $\bar{x} \in \Sigma=\{ x \in \Omega : f(x)=0 \}$ such that $Df(\bar{x}):\G\to\M$ is an h-epimorphism. Let $\W= \ker Df(\bar{x})$ and assume that $\V\subset\G$ is a homogeneous subgroup complementary to $\W$.
	Then there exist two open sets $\Omega' \subset \G, \ A \subset \W$ and a unique map $\phi:A   \to \V$ such that
	$$ \mathrm{graph}(\phi)=\Sigma \cap \Omega'.$$
	Moreover, there exists a constant $K>0$ such that, for every $w,w' \in A$ \begin{equation}\label{intlippar}
		\| \phi(w')^{-1} \phi(w) \| \leq K \| \phi(w')^{-1} w'^{-1} w \phi(w') \|.
	\end{equation}
\end{teo}

\begin{Remark}\label{rem:iLipc}
According to the definition of intrinsic Lipschitz function, taking into account \eqref{intlipWnormal}, condition \eqref{intlippar} can be rephrased saying that $\phi$ is $K$-intrinsic Lipschitz.
\end{Remark}


\subsection{Intrinsic regularity of $(\G,\M)$-regular sets of $\G$}\label{sect:levelsets}

In this section, we present a stronger version of Theorem~\ref{IFT}, where
the implicit map is uniformly intrinsically differentiable and the level
set is represented with respect to a suitably large class of factorizations.
We finally introduce the notion of $(\G,\M)$-regular set of $\G$.

The following theorem essentially corresponds to \cite[Theorem~4.3.7]{CorniPhD}, where the first equality of \eqref{eq:T_xkerDf} can be seen as a consequence of Theorem~\ref{characterizationP}.

\begin{teo}[Uniform intrinsic differentiability of level sets] 	\label{intdiffpar}
Let $\G$ and $\M$ be stratified groups, and let $\Omega \subset \G$ be an open set. We consider $f \in C^1_h(\Omega, \M)$ and
define $\Sigma=f^{-1}(0)$. Let us assume that there are an open set $\Omega' \subset \G$ 
such that $\Omega'\cap\Sigma\neq\emptyset$ and a homogeneous subgroup $\V \subset \G$  such that $Df(x)|_{\V} : \V \to \M$ is an $h$-isomorphism for every $x \in \Sigma \cap \Omega'$.
By our assumptions, we can consider a fixed couple of complementary subgroups $(\W,\V)$ of $\G$.
We also assume that we have an open set $A \subset \W$, $\phi: A \to \V$ with graph map $\Phi(w)=w\phi(w)$, such that $\Sigma \cap \Omega'= \Phi(A)$.
Then $\phi$ is uniformly intrinsically differentiable at any point of $A$ and
\beq\label{eq:T_xkerDf}
\T_{\Phi(w)}=\mathrm{graph}(d\phi_w)=\ker Df(\Phi(w))
\eeq
for every $w\in A$.
\end{teo}

Combining Theorem~\ref{intdiffpar} and Theorem~\ref{IFT} we immediately obtain the following
stronger version of the implicit function theorem.

\begin{teo}[Implicit function theorem]\label{corintdiffparam}
Let $\G$ and $\M$ be stratified groups, and let $\Omega \subset \G$ be an open set. Let $f \in C^1_h(\Omega, \M)$ and set $\Sigma=f^{-1}(0)$. For $\bar{x} \in \Sigma$ we assume that there exists a homogeneous subgroup $\V$ of $\G$ such that $Df(\bar{x})|_{\V}:\V\to\M$ is invertible. We can fix an arbitrary couple 
of complementary subgroups $(\W,\V)$. Then there exist an open neighbourhood $\Omega' \subset \Omega$ of $\bar{x}$ and a unique uniformly intrinsically differentiable map $\phi:A \to \V$, where $A \subset \W$ is open, $\pi^{\W,\V}_{\W}(\bar{x}) \in A$, such that $\Sigma \cap \Omega'=\mathrm{graph}(\phi)$. 
In addition, defining the graph map $\Phi(w)=w\phi(w)$,  
for every $w \in \Phi^{-1}(\Sigma \cap \Omega')$ we have $\mathrm{graph}(d\phi_w)=\ker Df(\Phi(w))$.
\end{teo}
\begin{proof}
Set $\elle=\ker Df(\bar{x})$ and $\bar{\ell}=\pi_{\elle}^{\elle,\V}(\bar{x})$.
By Theorem~\ref{IFT} there exists a unique map $\psi:A' \to \V$, where $A' \subset \elle$ is an open set with $\bar{\ell} \in A'$, and there exists an open set $\Omega' \subset\Omega$ such that $$\mathrm{graph}(\psi)=\Sigma \cap \Omega'.$$ On the other side, for every $\ell \in \elle$ we may consider $\ell=\pi_{\W}^{\W,\V}(\ell) \pi_{\V}^{\W,\V}(\ell)$, hence $$\pi_{\W}^{\W,\V}(\ell)= \ell \pi_{\V}^{\W,\V}(\ell)^{-1}.$$
Consequently, applying the projection $\pi_{\V}^{\elle,\V}$ we get that $\pi_{\V}^{\elle,\V}(\pi_{\W}^{\W,\V}(\ell))=\pi_{\V}^{\W,\V}(\ell)^{-1}$ and then surely $$\pi_{\V}^{\W,\V}(\ell)=(\pi_{\V}^{\elle,\V}(\pi_{\W}^{\W,\V}(\ell)))^{-1}.$$
Thus, for every $\ell \in A'$ we can set $w=\pi_{\W}^{\W,\V}(\ell)$ and we have
\begin{align*}
	\ell \psi(\ell)&= \pi_{\W}^{\W,\V}(\ell) \pi_{\V}^{\W,\V}(\ell) \psi(\pi_{\W}^{\W,\V}(\ell) \pi_{\V}^{\W,\V}(\ell)) \\
	&= \pi_{\W}^{\W,\V}(\ell)(\pi_{\V}^{\elle,\V}(\pi_{\W}^{\W,\V}(\ell)))^{-1} \psi(\pi_{\W}^{\W,\V}(\ell) (\pi_{\V}^{\elle,\V}(\pi_{\W}^{\W,\V}(\ell)))^{-1}) \\
	&= w(\pi_{\V}^{\elle,\V}(w))^{-1} \psi(w (\pi_{\V}^{\elle,\V}(w))^{-1}).
\end{align*}
Therefore, setting $A=\pi_{\W}^{\W,\V}(A')$ and $\phi:A \to \V$, $\phi(w)=(\pi_{\V}^{\elle,\V}(w))^{-1} \psi(w (\pi_{\V}^{\elle,\V}(w))^{-1})$ for every $w \in A$, we have
$$\mathrm{graph}(\phi)=\mathrm{graph}(\psi)=\Sigma \cap \Omega'.$$
Moreover, we observe that $\pi_{\W}^{\W,\V}(\bar{\ell})=\pi_{\W}^{\W,\V}(\bar{x}) \in \pi^{\W,\V}_\W(A')=A$.
Now, let us notice that the continuity of $x\to Df(x)$ ensures the existence of an open neighbourhood $\widetilde{\Omega} \subset \Omega$ of $\bar{x}$ such that $Df(x)|_{\V}$ is invertible for every $x \in \widetilde{\Omega}$. Finally, we can apply Theorem~\ref{intdiffpar}, where $\Omega'$ is replaced 
with $\Omega'\cap\widetilde{\Omega}$.
\end{proof}

\begin{Remark}\label{rem:exglobA}
Notice that if in the assumptions of Theorem~\ref{corintdiffparam} 
we also assume that for some $\Omega'\subset\Omega$ we have $\Omega' \subset \Omega\neq\emptyset$ and $J_\V f(y)> 0$ for any $y \in \Sigma \cap \Omega'$, then there exists a unique map $\phi:A\to\V$ such that
its graph mapping $\Phi:A\to\G$ satisfies $\Sigma \cap \Omega' = \Phi(A)$, where $A=\pi^{\W,\V}_\W(\Sigma \cap \Omega')\subset\W$ is an open set. This is an easy consequence of existence and uniqueness of the local graph mapping, proved in Theorem~\ref{corintdiffparam}.
\end{Remark}

From \cite[Definition~3.5]{Mag5} and \cite[Definition~10.2]{Mag14}, level sets of differentiable mappings, with specific conditions on the differential, constitute the ``intrinsic regular sets''. We recall this notion of regularity in the next definition.

\begin{deff}\label{def:GMregular}
	Let $\G$ and $\M$ be stratified groups. A subset $S\subset \G$ is a \textit{$(\G,\M)$-regular set of $\G$} if for every point $\bar{x} \in S$ there exist an open neighbourhood $\Omega \subset \G$ of $\bar{x}$
	and a continuously differentiable map $f: \Omega \to \M$ such that $\Sigma \cap \Omega=f^{-1}(0)$ and $Df(x): \G \to \M$ is an h-epimorphism
	for every $x\in\Omega$. 
\end{deff}

\begin{Remark}
	An important consequence of Theorem~\ref{corintdiffparam} is that all $(\G,\M)$-regular sets of $\G$ can be locally parametrized by uniformly intrinsically differentiable maps.
\end{Remark}

Notice that if $Df(z): \G \to \M$ is an h-epimorphism and $f$ is continuously
differentiable, then Proposition~\ref{char-hepi} gives us a homogeneous subgroup $\V\subset\G$ such that $Df(\bar x)|_{\V}$ is invertible. 
Thus, the continuity of $x\to Df(x)$ on $\Omega$ gives an open subset $\Omega' \subset \Omega$ containing $z$ such that $Df(x)|_{\V}$ is invertible for every $x \in \Omega'$. 
The same proposition shows that $Df(x)$ is an h-epimorphism for every $x\in\Omega'$. 
We have shown that the set of points for which the differential is an h-epimorphism is open in $\G$.

We conclude this section by briefly mentioning that a ``dual class'' of 
regular sets is also available, corresponding to images
of differentiable mappings. These sets are called $(\M,\G)$-regular sets of $\G$, 
see \cite[Definition~10.3]{Mag14}. They can be locally represented as 
intrinsic graphs as well, by the rank theorem proved in \cite[Theorem~1.5]{Mag14}.

\section{Algebraic lemma about the Jacobian of projections}
\label{sect:jacproj}

The section is entirely devoted to establishing Lemma~\ref{lm:MVPsubgroups}.
The main point is to compute the Jacobian of a projection with respect to
two couples of subgroups $(\W,\V)$ and $(\U,\V)$ of $\G$. The key aspect is that the second factor in the two factorizations does not change. 

We point out that our computations are carried out
with respect to a fixed scalar product on $\G$, that automatically induces a scalar product on $\Lambda_{q}(\G)$.
In the sequel, associating a simple $k$-vector to a $k$-dimensional vector space will be useful. 
Let $V \subset \G$ be a $k$-dimensional subspace and denote by $ \textbf{V} \in \Lambda_k \G\sm\set{0}$ an
\textit{orienting $k$-vector for $V$}, i.e.\ a simple $k$-vector $\textbf{V}$ such that $V=\{ v \in \G : \textbf{V} \wedge v=0 \}$.
We observe that when $|\textbf{V}|=1$ and an arbitrary orthonormal basis $(v_1, \dots , v_k)$ of $V$ is fixed, then
$\textbf{V}=\pm v_1 \wedge \dots \wedge v_k$.

\begin{lem}[Algebraic lemma]\label{lm:MVPsubgroups}
Let $\V\subset \G$ be a $p$-dimensional homogeneous subgroup and consider two homogeneous subgroups $\U$, $\W$ such that $(\W,\V)$ and $(\U,\V)$ are couples of complementary subgroups of $\G$.
Let $\bV$ be an orienting unit $p$-vector of $\V$ and let $\bW$, $\bU$ be orienting unit $(q-p)$-vectors of $\W$ and $\U$, respectively.
Then for every Borel set $B\subset \U$, we have
$$ (\pi_{\U,\W}^{\mathbb{U,V}})_\sharp \mathcal{H}_{|\cdot|}^{q-p} (B)=
\mathcal{H}_{|\cdot|}^{q-p}(\pi_{\W,\mathbb{U}}^{\mathbb{W,V}}(B))= \frac{| \bV \wedge \mathbf{U} |}{|\bV \wedge \bW |} \mathcal{H}^{q-p}_{|\cdot|}(B),
$$
where the projections $\pi^{\U,\V}_{\U,\W}$ and $\pi^{\U,\V}_{\W,\U}$ have been introduced in Definition~\ref{d:projWVM}. 
\end{lem}
\begin{proof}
We set $\ell_i=\mathrm{dim}(\V \cap V_i)$ and notice that 
\[
m_i= \mathrm{dim}(\W \cap V_i)= \mathrm{dim}(\U \cap V_i)
\] 
for every $i=1, \dots, \iota$. Surely $0 \leq \ell_i, m_i \leq n_i$ and $m_i=n_i-\ell_i$. 
Moreover $\sum_{i=1}^{\iota} \ell_i=p$, $\sum_{i=1}^{\iota} m_i=q-p$. 
Let us introduce an orthonormal graded basis 
\[ 
(v_1, \ldots, v_{\ell_1}, v_{h_1+1}, \ldots, v_{h_1+\ell_1}, \ldots, v_{h_{\iota-1}+1}, \ldots, v_{h_\iota})
\]
of $\V$ such that
\begin{align*}
\set{v_1, \dots, v_{\ell_1}} &\subset \V \cap V_1\\
\set{v_{h_1+1}, \dots, v_{h_1+\ell_2}} &\subset \V \cap V_2\\
 & \ldots \\
\set{v_{h_{\iota-1}+1}, \dots, v_{h_{\iota-1}+\ell_{\iota}}} &\subset \V \cap V_\iota.
\end{align*}
We complete the previous vectors to an orthonormal graded basis of $\G$. Then for each $j=1, \dots, \iota$, we choose $m_j$ vectors $v_{h_{j-1}+\ell_j+1}, \dots, v_{h_j} \in V_j \setminus (\V \cap V_j)$ so that $(v_1, \dots, v_q)$ is a graded orthonormal basis of $\G$.
 We introduce the following set of indices $$I_{\V}= \{ n \in \{1, \dots, q \}: \text{there exists} \ v_n\in\V \}.$$
By $j \notin I_{\V}$, we mean $j \in \{ 1, \dots, q \} \setminus I_{\V}$.
Let us choose two orthonormal graded bases $(w_i)_{i \notin I_{\V}}$ and $(u_i)_{i \notin I_{\V}}$ of $\W$ and $\U$, respectively, such that
\begin{align*}
\set{w_{\ell_1+1}, \dots, w_{h_1}} \subset \W \cap V_1,& \qquad  \set{u_{\ell_1+1}, \dots, u_{h_1}} \subset \U \cap V_1, \\
\set{w_{h_1+\ell_2+1}, \dots w_{h_2}} \subset \W \cap V_2,& \qquad \set{u_{h_1+\ell_2+1}, \dots u_{h_2}} \subset \U \cap V_2, \\
&\!\!\cdots \\
\set{w_{h_{\iota-1}+\ell_{\iota}+1}, \dots, w_\iota} \subset \W \cap V_\iota, & \qquad  \set{u_{h_{\iota-1}+\ell_{\iota}+1}, \dots, u_\iota} \subset \U \cap V_\iota.
\end{align*}
Let us introduce the linear isomorphism $i_\G:\G\to\R^q$,
\[
i_\G\pa{\sum_{i=1}^q x_i v_i}=(x_1, \dots, x_{q}).
\]
Since $(v_1, \dots, v_q)$ is an orthonormal basis of $\G$, the map $i_{\G}$ is an isometry between $\G$, equipped with our fixed scalar product $ \langle \cdot, \cdot \rangle$, and $\R^q$, equipped with the standard Euclidean scalar product. We define also the following isometries $i_\W : \mathbb{W} \to \mathbb{R}^{q-p}$, 
\[
i_\W\pa{\sum_{i \notin I_{\V}} x_iw_i }=(x_{\ell_1+1}, \ldots, x_{h_1}, x_{h_1+\ell_2+1} ,\ldots, x_{h_2}, \dots, x_{h_{\iota-1}+\ell_{\iota}+1}, \dots, x_{q}),
\]
$i_\U: \U\to \mathbb{R}^{q-p}$, 
\[
i_\U\pa{\sum_{i\notin I_{\V}} x_i u_i}=(x_{\ell_1+1}, \ldots, x_{h_1}, x_{h_1+\ell_2+1} ,\ldots, x_{h_2}, \ldots, x_{h_{\iota-1}+\ell_{\iota}+1}, \ldots, x_{q}).
\]
Since $\W$ and $\V$ are complementary, the smooth
map $\Psi_1:\mathbb{R}^{q}\to\G$,
\begin{equation}\label{eq:Psi1}
\Psi_1(x_1, \dots , x_{q})=	\bigg(\sum_{i \notin I_{\V}} x_i w_i \bigg) \bigg( \sum_{i \in I_{\V}} x_i v_i \bigg)
\end{equation}
is a diffeomorphism.
To denote wedge products of ordered sets of vectors we use the notation
$$
\bigwedge_{i \in I}x_i=x_{i_1} \wedge \ldots x_{i_m}, \mathrm{\  with \ } i_1, \ldots, i_m \in I \mathrm{ \ and \ }i_1 < i_2 < \ldots < i_m.
$$
where $I \subset \N^m$ for some $m \in \N$ and $\{x_i\}_{i \in I} \subset \G$  is a set of vectors.
From our definitions, we have
\[
\bV=\pm\bigwedge_{i \in I_\V}v_i,\quad \bW=\pm \bigwedge_{i \notin I_\V}w_i,\quad
\bU=\pm\bigwedge_{i \notin I_\V}u_i.
\]
Thus, from the BCH formula, the special form of the 
partial derivatives of $\der_{x_i}\Psi_1$ leads us
to the equality 
$J \Psi_1(x)= | \bV \wedge \bW |$ for every $x \in \R^q$.
Let us consider the composition $$\widetilde{\Psi}_1=i_{\G} \circ \Psi_1: \R^q \to \R^q.$$
Since the map $i_{\G}$ is an isometry, the equalities
\begin{equation}
\label{eq:ugjac}
J \widetilde{\Psi}_1(x) =J \Psi_1(x)= | \bV \wedge \bW |
\end{equation}
hold for every $x \in \R^q$, where the Jacobians are understood with respect to the corresponding metrics. 
We denote by $\nabla\widetilde{\Psi}_1(x)$ the Jacobian matrix of $\widetilde \Psi_1$ with respect to the fixed scalar product 
of $\G$ for $x\in \R^q$.  
From the BCH formula \eqref{eq:BCH}, we get 
\begin{align*}
&\Psi_1(x_1, \dots, x_q)= \bigg(\sum_{i \notin I_{\V}} x_i w_i \bigg) \bigg( \sum_{i \in I_{\V}} x_i v_i \bigg)\\
&= \sum_{i \notin I_{\V}} x_i w_i +  \sum_{i \in I_{\V}} x_i v_i + \sum_{s=2}^\iota \sum_{j=h_{s-1}+1}^{h_s} Q_j\left(\sum_{i \notin I_{\V}, \ i \leq h_{s-1}} x_i w_i, \sum_{i \in I_{\V}, \ i \leq h_{s-1}} x_i v_i\right)v_j\\
&=\sum_{i \notin I_{\V}, i \leq h_1} x_i w_i +  \sum_{i \in I_{\V}, \ i \leq h_1} x_i v_i  \\
& \ + \sum_{i \notin I_{\V}, h_1+1 \leq i \leq h_2} x_i w_i +  \sum_{i \in I_{\V}, \ h_1+1 \leq i \leq h_2} x_i v_i + \sum_{j=h_{1}+1}^{h_2} Q_j\left(\sum_{i \notin I_{\V}, \ i \leq h_{1}} x_i w_i, \sum_{i \in I_{\V}, \ i \leq h_{1}} x_i v_i\right)v_j\\
& \ + \sum_{i \notin I_{\V}, h_2+1 \leq i \leq h_3} x_i w_i +  \sum_{i \in I_{\V}, \ h_2+1 \leq i \leq h_3} x_i v_i + \sum_{j=h_{2}+1}^{h_3} Q_j\left(\sum_{i \notin I_{\V}, \ i \leq h_{2}} x_i w_i, \sum_{i \in I_{\V}, \ i \leq h_{2}} x_i v_i\right)v_j\\
& \ + \dots   \\
& \ +\sum_{i \notin I_{\V}, h_{\iota-1}+1 \leq i \leq h_\iota} x_i w_i +  \sum_{i \in I_{\V}, \ h_{\iota-1}+1 \leq i \leq h_\iota} x_i v_i + \sum_{j=h_{\iota-1}+1}^{h_\iota} Q_j\left(\sum_{i \notin I_{\V}, \ i \leq h_{\iota-1}} x_i w_i, \sum_{i \in I_{\V}, \ i \leq h_{\iota-1}} x_i v_i\right)v_j.
\end{align*}
For every $k=1, \dots, \iota$ and for every integer $i$ with $ h_{k-1} + \ell_k +1 \leq i \leq h_{k}$, we consider
$$
w_i= \sum_{r \in \{ h_{k-1} +1, \dots, h_{k} \} } f_r^i v_r
$$ 
for some coefficients $f_r^i \in \R$. Thus we continue from the above expression of $\Psi_1$, getting
\begin{align*}
\Psi_1(x_1, \dots, x_q)=  & \sum_{i \notin I_{\V}, i \leq h_1} x_i \Big( \sum_{r \in \{1, \dots, h_{1} \} } f_r^i v_r\Big) +  \sum_{i \in I_{\V}, \ i \leq h_1} x_i v_i\\ 
&+\sum_{s=2}^{\iota} \Bigg( \sum_{i \notin I_{\V}, h_{s-1}+1 \leq i \leq h_s} x_i\Big( \sum_{r \in \{h_{s-1}+1, \dots, h_{s} \} } f_r^i v_r\Big) +  \sum_{i \in I_{\V}, \ h_{s-1}+1 \leq i \leq h_s} x_i v_i  \\
& +\sum_{j=h_{s-1}+1}^{h_s} Q_j\left(\sum_{i \notin I_{\V}, \ i \leq h_{s-1}} x_i w_i, \sum_{i \in I_{\V}, \ i \leq h_{s-1}} x_i v_i\right) v_j \Bigg)
\end{align*}
and we can reorganize the previous equation as follows
\begin{align*}
\Psi_1(x_1, \dots, x_q)=& \sum_{r \notin I_{\V}, \ r \leq h_1}  \Big(\sum_{i \notin I_{\V}, i \leq h_1} x_i f_r^i \Big)v_r +  \sum_{r \in I_{\V}, \ r \leq h_1} \Big( x_r + \sum_{i \notin I_{\V}, \ i \leq h_1} x_i f_r^i \Big) v_r  \\
& +  \sum_{s=2}^{\iota} \Bigg(\sum_{r \notin I_{\V}, h_{s-1}+1 \leq  r \leq h_s} \Bigg(\sum_{i \notin I_{\V},  h_{s-1}+1 \leq  i \leq h_s} x_i f_r^i \\
&\qquad \qquad \qquad \qquad + Q_r\left(\sum_{i \notin I_{\V}, \ i \leq h_{s-1}} x_i w_i, \sum_{i \in I_{\V}, \ i \leq h_{s-1}} x_i v_i\right)    \Bigg) v_r \stepcounter{equation}\tag{\theequation}\label{reorg2}\\
& + \sum_{r \in I_{\V}, h_{s-1}+1 \leq  r \leq h_s} \Bigg( x_r + \sum_{i \notin I_{\V}, h_{s-1}+1 \leq  i \leq h_s} x_i f_r^i \\
&\qquad \qquad \qquad \qquad +   Q_r\left(\sum_{i \notin I_{\V}, \ i \leq h_{s-1}} x_i w_i, \sum_{i \in I_{\V}, \ i \leq h_{s-1}} x_i v_i\right)   \Bigg) v_r \Bigg).
\end{align*}
For every $s=2, \dots \iota$ and $j=h_{s-1}+1, \dots, h_{s}$, we denote by $\widetilde{Q}_j : \R^q \to \R$ the following map
\begin{align*}
\widetilde{Q}_j(x_1, \dots, x_q)&= Q_j\left( \sum_{i \notin I_{\V}, i \leq h_{s-1}} x_iw_i, \sum_{i \in I_{\V}, i \leq h_{s-1}} x_i v_i\right).
\end{align*}
Now, we rewrite \eqref{reorg2} with respect to the basis $(v_1, \dots, v_q)$, getting
\begin{align*}
\Psi_1(x_1, \dots, x_q)&=\sum_{i=1}^{\ell_1}\left( x_i+ \sum_{j=\ell_1+1}^{h_1} x_j f_i^j\right) v_i + \sum_{i=\ell_1+1}^{h_1} \left(\sum_{j=\ell_1+1}^{h_1}x_jf_i^j\right)v_i\\
&+\sum_{s=2}^\iota  \sum_{i=h_{s-1}+1}^{h_{s-1}+\ell_s} \left(x_i+ \sum_{j=h_{s-1}+\ell_s+1}^{h_s} x_j f^j_i + \widetilde{Q}_i(x) \right)v_i\\
&+\sum_{s=2}^\iota\sum_{i=h_{s-1}+\ell_s+1}^{h_s}\left(\sum_{j=h_{s-1}+\ell_s+1}^{h_s}x_jf^j_i+  \widetilde{Q}_i(x) \right)v_i.
\end{align*}
As a consequence, denoting by $(e_1, \dots, e_q)$ the canonical basis of $\R^q$, we have
\begin{eqnarray*}
&&\widetilde{\Psi}_1(x_1, \dots, x_q)=\sum_{i=1}^{\ell_1}\left( x_i+ \sum_{j=\ell_1+1}^{h_1} x_j f_i^j\right) e_i + \sum_{i=\ell_1+1}^{h_1} \left(\sum_{j=\ell_1+1}^{h_1}x_jf_i^j\right)e_i\\
&&\!\!\!\!\!\!\!\!+\sum_{s=2}^\iota \Bigg( \sum_{i=h_{s-1}+1}^{h_{s-1}+\ell_s} \left(x_i+ \!\!\!\!\sum_{j=h_{s-1}+\ell_s+1}^{h_s} x_j f^j_i + \widetilde{Q}_i(x) \right)e_i
+\!\!\!\!\sum_{i=h_{s-1}+\ell_s+1}^{h_s}\!\!\!\left(\sum_{j=h_{s-1}+\ell_s+1}^{h_s}x_jf^j_i+  \widetilde{Q}_i(x) \right)e_i \bigg).
\end{eqnarray*}
From the form of $\widetilde \Psi_1$, the Jacobian matrix $\nabla \widetilde{\Psi}_1(x)$ is the block lower triangular matrix
\begin{equation}\label{shapeofthematrix}
\nabla \widetilde{\Psi}_1(x)=\begin{bmatrix}
S_{n_1} & 0_{n_1, n_2} & 0_{n_1, n_3} & \dots & 0_{n_1,n_\iota}\\
M_{n_2, n_1}(x) & S_{n_2} & 0_{n_2, n_3} & \dots &  0_{n_2,n_\iota}\\
M_{n_3, n_1}(x) & M_{n_3, n_2}(x) & S_{n_3} & \dots &  0_{n_3,n_\iota}\\
\dots & \dots & \dots& \dots & \dots \\
M_{n_\iota, n_1}(x) & M_{n_\iota, n_2}(x) & \dots & M_{ n_\iota, n_{\iota-1}}(x) &  S_{n_\iota}\\
\end{bmatrix}.
\end{equation}
For $k,j \in \{1, \dots \iota \}$ the symbol $0_{k,j}$ denotes the $k \times j$ zero matrix. The matrices $S_{n_k}$, $k \in \{ 1 , \dots, \iota \}$, are the following $n_k \times n_k$ block upper triangular matrices
$$S_{n_k}= \begin{bmatrix}
1 & \dots & 0  & f_{h_{k-1}+1}^{h_{k-1}+\ell_k+1} & f_{h_{k-1}+1}^{h_{k-1}+\ell_k+2} & \dots & f_{h_{k-1}+1}^{h_k}  \\
\dots & \dots & \dots  & \dots & \dots & \dots & \dots  \\
0 & \dots & 1  & f_{h_{k-1}+i_{k}}^{h_{k-1}+\ell_k+1} & f_{h_{k-1}+\ell_k}^{h_{k-1}+\ell_k+2} & \dots & f_{h_{k-1}+\ell_k}^{h_k}  \\
0 & \dots & 0  & f_{h_{k-1}+\ell_k+1}^{h_{k-1}+\ell_k+1} & f_{h_{k-1}+\ell_k+1}^{h_{k-1}+\ell_k+2} & \dots & f_{h_{k-1}+\ell_k+1}^{h_k}  \\
\dots & \dots & \dots  & \dots & \dots & \dots & \dots  \\
0 & \dots & 0  & f_{h_k}^{h_{k-1}+\ell_k+1} & f_{h_k}^{h_{k-1}+\ell_k+2} & \dots & f_{h_k}^{h_k} \end{bmatrix}=\begin{bmatrix}
\mathbb{I}_{\ell_k} & F_{k} \\
0_{m_k, \ell_k} & R_k\\
\end{bmatrix},
$$
where clearly $R_k$ and $F_k$ are a $m_k \times m_k$ and a $\ell_k \times m_k$ real matrix, respectively. 
The matrices $M_{n_k, n_j}(x)$ for $k, j \in \{1, \dots \iota \}$ depend on the point $x=(x_1, \dots x_q)$ 
$$ M_{n_k, n_j}(x)= \begin{bmatrix}
\frac{\partial{\widetilde{Q}_{h_{k-1}+1}}}{\partial x_{h_{j-1}+1}}(x)  & \dots & \frac{\partial{\widetilde{Q}_{h_{k-1}+1}}}{\partial x_{h_j}}(x)  \\
\dots & \dots & \dots\\
 \frac{\partial{\widetilde{Q}_{h_k}}}{\partial x_{h_{j-1}+1}}(x) & \dots & \frac{\partial{\widetilde{Q}_{h_k}}}{\partial x_{h_j}}(x) \\
\end{bmatrix}.$$
We define another map $\Psi_2:\mathbb{R}^{q}\to \G$,
\beq\label{eq:Psi2}
\Psi_2(x_1, \dots , x_{q})=	\bigg(\sum_{i \notin I_{\V}} x_i u_i \bigg) \bigg( \sum_{i \in I_{\V}} x_i v_i \bigg).
\eeq
Again, by direct verification one can check that $J \Psi_2(x)= | \bV \wedge \bU |$ for every $x \in \R^q$. We set $$ \widetilde{\Psi}_2= i_{\G} \circ \Psi_2: \R^q \to \R^q,$$ hence arguing as in the case of $\widetilde{\Psi}_1$, we get
\begin{equation}\label{eq:ugjac2}
J \widetilde{\Psi}_2(x)= J \Psi_2(x)= | \bV \wedge \bU |
\end{equation}
for every $x \in \R^q$. We notice that replacing $u_i$ with $w_i$ in \eqref{eq:Psi2} we get exactly the expression of $\Psi_1$ in \eqref{eq:Psi1}. 
Then the form of the Jacobian matrix $\nabla \widetilde{\Psi}_2(x)$ at any point $x \in \R^q$ is the same as in \eqref{shapeofthematrix}.
More precisely, only the block matrices 
$F_k$, $R_k$, $M_{n_i,n_j}$, for $i,j,k \in \{1, \dots, \iota \}$, change. 
Taking into account that
$\widetilde{\Psi}_1^{-1}\circ \widetilde{\Psi}_2=\Psi_1^{-1} \circ \Psi_2$, we get the equalities
\begin{equation}\label{inversa}
	\begin{split}
	\nabla(\Psi_1^{-1} \circ \Psi_2)(x)=\nabla( \widetilde{\Psi}_1^{-1} \circ \widetilde{\Psi}_2)(x)&= \nabla \widetilde{\Psi}_1^{-1}(\widetilde{\Psi}_2(x))\nabla \widetilde{\Psi}_2(x)\\
	&=(\nabla\widetilde{\Psi}_1(\widetilde{\Psi}_1^{-1}( \widetilde{\Psi}_2(x)))^{-1} \nabla \widetilde{\Psi}_2(x)
	\end{split}
\end{equation}
for every $x \in \R^q$.
Thus, combining \eqref{eq:ugjac} and \eqref{eq:ugjac2}, the equalities
\begin{equation}\label{eq:jacmet}
\begin{split}
J( \Psi_1^{-1} \circ \Psi_2)(x)&= J( \widetilde{\Psi}_1^{-1} \circ \widetilde{\Psi}_2)(x)=(J \widetilde{\Psi}_1((\widetilde{\Psi}_1^{-1}\circ \widetilde{\Psi}_2)(x)))^{-1} J \widetilde{\Psi}_2(x)\\
&=(J \Psi_1((\Psi_1^{-1} \circ \Psi_2)(x)))^{-1} J \Psi_2(x)=\frac{| \bV \wedge \bU |}{| \bV \wedge \bW| }
\end{split}
\end{equation}
hold for every $x \in \R^q$.
It is easy to observe that all real and invertible $q \times q$ block lower triangular matrices whose diagonal blocks are block upper triangular matrices constitute a Lie subgroup of $GL(q,\R)$. For this reason, \eqref{inversa} and \eqref{shapeofthematrix} ensure that $\nabla( \Psi_1^{-1} \circ \Psi_2)(x)$ is the $q \times q$ block lower triangular matrix
\begin{equation}\label{eq:matrix}
\nabla ( \Psi_1^{-1} \circ \Psi_2)(x)=\begin{bmatrix}
\widehat{S}_{n_1} & 0_{n_1, n_2} & 0_{n_1, n_3} & \dots & 0_{n_1,n_\iota}\\
\widehat{M}_{n_2, n_1}(x) & \widehat{S}_{n_2} & 0_{n_2, n_3} & \dots &  0_{n_2,n_\iota}\\
\widehat{M}_{n_3, n_1}(x) & \widehat{M}_{n_3, n_2}(x) & \widehat{S}_{n_3} & \dots &  0_{n_3,n_\iota}\\
\dots & \dots & \dots& \dots & \dots \\
\widehat{M}_{n_\iota, n_1}(x) & \widehat{M}_{n_\iota, n_2}(x) & \dots & \widehat{M}_{ n_\iota, n_{\iota-1}}(x) &  \widehat{S}_{n_\iota}\\
\end{bmatrix},
\end{equation}
where $\widehat{M}_{n_j, n_k}(x)$ are real $n_j\times n_k$ matrices and $\widehat{S}_{n_k}$ are $n_k\times n_k$ matrices of the form
\begin{equation}\label{eq:sk}
\widehat{S}_{n_k}=\begin{bmatrix}
\mathbb{I}_{\ell_k} & \widehat{F}_{k} \\
0_{m_k, \ell_k} & \widehat{R}_k\\
\end{bmatrix},
\end{equation}
with $j, k \in \{ 1 , \dots, \iota \}$.
We have denoted by $\widehat{R}_k$ and $\widehat{F}_k$ an $m_k \times m_k$ matrix and an $\ell_k \times m_k$ matrix, respectively. 
Since \eqref{eq:matrix} is block lower triangular, its determinant has the form
$$ J( \Psi_1^{-1} \circ \Psi_2)(x)=|\det (\nabla( \Psi_1^{-1} \circ \Psi_2)(x))|=| \det(\widehat{S}_1) \det(\widehat{S}_2) \cdots \det( \widehat{S}_{\iota})|.$$
Notice that it is independent of $x$.
On the other side, since the matrices $\widehat{S}_{n_k}$ block upper triangular in \eqref{eq:sk}, we obtain the equality
\begin{equation}\label{eq:jacprodR}
J( \Psi_1^{-1} \circ \Psi_2)(x)=|\det(\widehat{R}_1) \det(\widehat{R}_2) \cdots \det( \widehat{R}_{\iota})|.
\end{equation}
We introduce the embedding $E:\mathbb{R}^{q-p} \to \mathbb{R}^{q}$ as the linear map associated with the $q \times (q-p)$ matrix
\begin{equation*}
\begin{bmatrix}
0_{\ell_1, m_1}  & 0_{\ell_1, m_2} & \cdots & \cdots  & 0_{\ell_1, m_{\iota}}\\
\mathbb{I}_{m_1}  & 0_{m_1, m_2}  & \cdots & \cdots  & 0_{m_1, m_{\iota}}\\
0_{\ell_2, m_1}  & 0_{\ell_2, m_2}  & \cdots & \cdots  & 0_{\ell_2, m_{\iota}}\\
0_{m_2, m_1}  & \mathbb{I}_{m_2}  & \cdots & \cdots  & 0_{m_2, m_{\iota}}\\
\cdots & \cdots  & \cdots & \cdots  & \cdots\\
\cdots & \cdots & \cdots & \cdots  & \cdots\\
0_{\ell_\iota, m_1}  & 0_{\ell_\iota, m_2}  & \cdots & \cdots  & 0_{\ell_\iota, m_{\iota}}\\
0_{m_\iota, m_1}  & 0_{m_\iota, m_2}& \cdots & \cdots  & \mathbb{I}_{ m_{\iota}}\\
\end{bmatrix}.
\end{equation*}
Taking into account \eqref{eq:Psi2} and the form of the previous matrix, it is easy to realize the identity 
\beq\label{eq:DecIU}
u=\Psi_2\circ E\circ i_\U(u)
\eeq
for every $u\in\U$. We also define the projection $G: \mathbb{R}^{q} \to \mathbb{R}^{q-p}$ as the linear map associated with the following $(q-p) \times q$ matrix
\begin{equation*}
\begin{bmatrix}
0_{m_1, \ell_1} & \mathbb{I}_{m_1} & 0_{m_1, \ell_2} & 0_{m_1, m_2} & \cdots & \cdots & 0_{m_1, \ell_\iota} & 0_{m_1, m_{\iota}}\\
0_{m_2, \ell_1} & 0_{m_2, m_1} & 0_{m_2, \ell_2} & \mathbb{I}_{m_2} & \cdots & \cdots  & 0_{m_2, \ell_\iota} & 0_{m_2, m_{\iota}}\\
\cdots & \cdots    & \cdots  & \cdots & \cdots & \cdots & \cdots  & \cdots \\
\cdots & \cdots    & \cdots  & \cdots & \cdots & \cdots & \cdots  & \cdots \\
0_{m_\iota, \ell_1} & 0_{m_\iota, m_1} & 0_{m_\iota, \ell_2} & 0_{m_\iota, m_2}  & \cdots & \cdots & 0_{m_\iota, \ell_\iota} & \mathbb{I}_{ m_{\iota}}\\
\end{bmatrix}.
\end{equation*}
Notice that $G \circ E (x)=x$ for every $x \in \R^{q-p}$.
Due to \eqref{eq:Psi1}, the form of the previous matrix yields
\[
i_\W^{-1}\circ G\circ \Psi_1^{-1}(z)=\pi_\W(z)
\]
for every $z\in\G$.
By \eqref{eq:DecIU}, if we take any $u \in \U$ we get
\begin{align*}
\pi_\W(u)&=i_\W^{-1}\circ G\circ \Psi_1^{-1}(u) \\
&=i_\W^{-1}\circ G\circ \Psi_1^{-1}\circ\Psi_2\circ E\circ i_\U(u) \stepcounter{equation}\tag{\theequation}\label{eq:decpiW}\\
&=i_\W^{-1}\circ G\circ \widetilde{\Psi}_1^{-1}\circ \widetilde{\Psi}_2\circ E\circ i_\U(u)=\pi^{\W,\V}_{\W,\U}(u).
\end{align*}
Now, by the definitions of $E$ and $G$, it is not difficult to verify that for every $x \in \R^{q-p}$ the Jacobian matrix
$\nabla (G \circ \Psi_1^{-1} \circ \Psi_2 \circ E)(x)$ is the $(q-p) \times (q-p)$ matrix that we obtain by deleting the rows and columns of indexes belonging to $I_{\V}$ from the matrix $\nabla(\Psi_1^{-1} \circ \Psi_2)(E(x))$, hence
\begin{equation*}
 \nabla (G \circ \Psi_1^{-1} \circ \Psi_2 \circ E)(x)=\begin{bmatrix}
\widehat{R}_1 & 0_{m_1, n_2} & 0_{m_1, n_3} & \dots & \dots & 0_{m_1,n_\iota}\\
\widehat{N}_{m_2, n_1}(E(x)) & \widehat{R}_2 & 0_{m_2, n_3} & \dots & \dots &  0_{m_2,n_\iota}\\
\dots & \dots & \dots& \dots &\dots & \dots \\
\dots & \dots & \dots& \dots &\dots & \dots \\
\widehat{N}_{m_\iota, n_1}(E(x)) & \widehat{N}_{m_\iota, n_2}(E(x)) & \dots & \dots & \widehat{N}_{ m_\iota, n_{\iota-1}}(E(x)) &  \widehat{R}_{\iota}\\
\end{bmatrix},
\end{equation*}
where $\widehat{N}_{i,j}(E(x))$ are suitable real $i \times j$ matrix.
Thus, independently of the point $x \in \R^{q-p}$, by \eqref{eq:jacprodR} and \eqref{eq:jacmet} we obtain
\beq\label{eq:formula}
\begin{split}
J(G\circ \Psi_1^{-1}\circ \Psi_2\circ E )(x)&=|\det(\widehat{R}_1) \det(\widehat{R}_2) \cdots \det( \widehat{R}_{\iota})|\\
&=J(\Psi_1^{-1}\circ\Psi_2)(E(x))=\frac{| \bV \wedge \bU |}{| \bV \wedge \bW |}.
\end{split}
\eeq
Finally, as a consequence of \eqref{eq:decpiW} and \eqref{eq:formula}
for every Borel set $B \subset \U$ we get
\begin{align*}
\mathcal{H}^{q-p}_{|\cdot|}(B) &= \mathcal{L}^{q-p}(i_{\U}(B))
= \frac{| \bV  \wedge \bW |}{| \bV \wedge \bU |}  \mathcal{L}^{q-p}(G (\Psi_1^{-1}(\Psi_2(E(i_{\U}(B)))))\\
&= \frac{| \bV  \wedge \bW |}{| \bV \wedge \bU |} \mathcal{H}^{q-p}_{|\cdot|}(i_{\W}^{-1}(G(\Psi_1^{-1}(\Psi_2(E(i_{\U}(B)))))= \frac{| \bV  \wedge \bW |}{| \bV \wedge \bU |} \mathcal{H}^{q-p}_{|\cdot|}(\pi_{\W, \U}^{\W, \V}(B)),
\end{align*}
concluding the proof.
\end{proof}

\section{Area formula for intrinsically differentiable graphs}

In this section we prove the area formula for intrinsic graphs 
arising from continuously intrinsically differentiable maps.
The central tool is the ``upper blow-up'', established in Theorem~\ref{theorem:ubu}, which represents the main result of the paper.

\subsection{Intrinsic Jacobian}
We start by studying the notion of intrinsic Jacobian for graphs of intrinsically differentiable maps. 

\begin{deff}
Let $(\W,\V)$ be a couple of complementary subgroups of $\G$. Let $n$ be the topological dimension of $\W$. Let $A \subset \W$ be an open set and let $\bar{w} \in A$. We consider a map $\phi:A  \to \V$ and assume that $\phi$ is intrinsically differentiable at $\bar{w}$. 
For the graph map of the intrinsic differential $d\phi_{\bar w}$, we use the notation 
$$
G(d\phi_{\bar{w}}):\W \to \G, \ G(d\phi_{\bar{w}})(w)=w d\phi_{\bar{w}}(w)
$$
for all $w\in\W$. Then we introduce the \textit{intrinsic Jacobian} of the graph map $\Phi:A\to\G$ at $\bar{w}$ as 
\begin{equation}\label{def:JacPhi}
J\Phi(\bar{w})=\frac{\mathcal{H}_{|\cdot|}^n(G(d\phi_{\bar{w}})(B))}{\mathcal{H}^n_{|\cdot|}(B)},
\end{equation}
where $\Phi:A \to \G$, $\Phi(w)=w\phi(w)$ for all $w\in A$ and $B \subset \W$ is a Borel set of positive measure.
\end{deff}

The next proposition, in particular, proves that the Jacobian defined in \eqref{def:JacPhi} does not depend on the choice of the 
Borel set $B\subset\W$.

\begin{prop}\label{propjac}
Let $(\W,\V)$ be a couple of complementary subgroups of $\G$. Let $A \subset \W$ be an open set, $\bar{w} \in A$ and let $p$ be the topological dimension of $\V$. 
Let $\Phi:A \to \G$ be the graph map of $\phi:A\to \V$. Assume that $\phi$ is intrinsically differentiable at $\bar{w}$ and set \begin{equation}\label{Uw}\U_{\bar{w}}=\mathrm{graph}(d\phi_{\bar{w}}).
\end{equation} 
If $\bV$ is an orienting unit $p$-vector of $\V$ and $\bW$, $\bU_{\bar{w}}$ are orienting unit $(q-p)$-vectors of $\W$ and $\U_{\bar{w}}$, respectively, then
$$J \Phi(\bar{w})=\frac{|\bV \wedge \bW|}{|\bV \wedge \bU_{\bar{w}}|}.$$
\end{prop}
\begin{proof}
By Proposition~\ref{P7} (i), we know that $(\U_{\bar{w}},\V)$ is a couple of complementary subgroups of $\G$.
We wish to prove that
\begin{equation}
\label{eq:dpr}
G(d\phi_{\bar{w}})(w)=\pi_{\U_{\bar{w}}, \W}^{\U_{\bar{w}}, \V}(w)
\end{equation}
for every $w\in\W$, where $G(d\phi_{\bar{w}})(w)=w d \phi_{\bar{w}}(w)$. It suffices to notice that 
$$w= (G(d\phi_{\bar{w}})(w)) ( d\phi_{\bar{w}}(w))^{-1},$$
with $w\in\W$, $G(d\phi_{\bar{w}})(w) \in \U_{\bar{w}}$ and $(d\phi_{\bar{w}}(w))^{-1} \in \V$, hence \eqref{eq:dpr} is proved.
By combining the definition of intrinsic Jacobian and of Lemma~\ref{lm:MVPsubgroups}, we get
\begin{equation}
J \Phi(\bar{w})=\frac{\mathcal{H}_{|\cdot|}^n(G(d\phi_{\bar{w}})(B))}{\mathcal{H}^n_{|\cdot|}(B)}=\frac{\mathcal{H}_{|\cdot|}^n(\pi_{\U_{\bar{w}}, \W}^{\U_{\bar{w}}, \V}(B))}{\mathcal{H}^n_{|\cdot|}(B)}= \frac{|\bV \wedge \bW|}{|\bV \wedge \bU_{\bar{w}}|}
\end{equation}
for every Borel set $B \subset \W$.
\end{proof}

\begin{prop}\label{cont-jac}
Let $(\W,\V)$ be a couple of complementary subgroups of $\G$. Let $A \subset \W$ be an open set and consider $\phi:A  \to \V$. Assume that $\phi$ is intrinsically differentiable at any point of $A$ with continuous intrinsic differential $d\phi:A \to \IL(\W,\V)$. Then $w \to J\Phi(w)$ is continuous on $A$.
\end{prop}

\begin{proof}
Let us consider $n \in A$ and set $\U_n=\mathrm{graph}(d \phi_n)$.
Let $p$ denote the topological dimension of $\V$. By Proposition~\ref{propjac} we have 
\begin{equation}\label{eq:formajac}
J\Phi(n)= \frac{|\bV \wedge \bW|}{|\bV \wedge \bU_{n}|},
\end{equation} where $\bV$ is an orienting unit $p$-vector of $\V$ and $\bW$, $\bU_n$ are orienting unit $(q-p)$-vectors of $\W$ and $\U_n$, respectively.
We set for $s=1, \dots, \iota$, $\ell_s=\mathrm{dim}(\V \cap V_s)$ and we introduce a graded basis
\begin{equation}\label{eq:base}
\begin{split}
(v_1, \ldots, v_{\ell_1},w_{\ell_1+1}, \ldots, w_{h_1}&,
 v_{h_1+1}, \ldots, v_{h_1+\ell_2}, w_{h_1+\ell_2+1}, \ldots w_{h_2}, \ldots,\\
& \ldots,  v_{h_{\iota-1}+1}, \ldots, v_{h_{\iota-1}+\ell_{\iota}}, w_{h_{\iota-1}+\ell_{\iota}+1}, \ldots, w_{h_{\iota}})
\end{split}
\end{equation}
of $\G$ such that
for every $s=1, \dots, \iota$
$$\set{v_{h_{s-1}+1}, \dots, v_{h_{s-1}+\ell_{s}}} \subset \V \cap V_{s} \ \ \ \ \  \set{w_{h_{s-1}+\ell_s+1}, \dots, w_{h_s}} \subset \W \cap V_s.$$
We introduce the following set of $q-p$ indices $$I_{\W}= \{ j \in \{1, \dots, q \}: \text{there exists} \ w_j\in \W\}.$$
Taking into account \eqref{eq:base}, it is immediate to observe that $$I_{\W}= \{ \ell_1+1, \ldots, h_1, h_1+\ell_2+1, \ldots, h_2, \ldots, h_{\iota-1}+\ell_{\iota}+1, \ldots, h_{\iota} \}.$$
The central point of the proof is to show that
\begin{equation}\label{claim-baseinbase}
\{ G(d\phi_n)(w_{i}) \}_{i \in I_{\W}} \mathrm{ \ is \ a \ basis \ of \ }\U_n.
\end{equation}
If this is true, then there exists $\lambda \in \R$ such that 
\begin{align}\label{eq:Glambda}
\bigwedge_{i \in I_{\W}} G(d\phi_n)(w_{i}) =\lambda  \bU_n.
\end{align}
Notice that we have used the notation 
$$
\bigwedge_{i \in I}x_i=x_{i_1} \wedge \ldots x_{i_m}, \mathrm{\  with \ } i_1, \ldots, i_m \in I \mathrm{ \ and \ }i_1 < i_2 < \ldots < i_m
$$ 
where if $I \subset \N^m$, for some $m \in \N$ and $\{x_i\}_{i \in I} \subset \G$ is a set of vectors.
Due to \eqref{eq:Glambda}, since we are dealing with unit multivectors, we get
\begin{align}
\label{eq:vectcont}
\left|\bV \wedge \frac{\bigwedge_{i \in I_{\W}} G(d\phi_n)(w_{i}) }{\| \bigwedge_{i \in I_{\W}} G(d\phi_n)(w_{i}) \|}\right|=|\bV\wedge \bU_n|.
\end{align}
The continuity of the intrinsic differential implies the continuity of $A \ni n \to G(d\phi_{n})(w)$ for every fixed $w \in \W$. 
Indeed, taking $n,n' \in A$ it follows that
\begin{align*}
\|G(d\phi_{n})(w)^{-1}G(d\phi_{n'})(w)\|&=
\|d\phi_n(w)^{-1}d\phi_{n'}(w)\|\\
&= \| w\| \|d\phi_n(\delta_{\frac{1}{\|w \|}}(w))^{-1}d\phi_{n'}(\delta_{\frac{1}{\|w\|}}(w))\|\\
&\le \|w\|\ccS_{\W,\V}(d\phi_n,d\phi_{n'}),
\end{align*}
where we exploited the homogeneity of intrinsically linear maps and the definition of the distance $\ccS_{\W,\V}$. 
The continuity of $A \ni n \to G(d\phi_{n})(w_i)$ immediately implies the continuity of the scalar function in \eqref{eq:vectcont}
with respect to $n\in A$. 
The form of the Jacobian $J\Phi(n)$ in \eqref{eq:formajac} gives our claim.

Finally, we are left to prove \eqref{claim-baseinbase}. It suffices to prove that $\{ P_{\W}(G(d\phi_n)(w_{i}))\}_{i \in I_{\W} } $ is a basis of $\W$, where the key fact is that $P_{\W}:\G \to \W$ is the linear projection associated with the direct sum $\W \oplus \V$.
Setting $L=d\phi_n$ and $i \in I_{\W} $, we compute  $G(d\phi_n)(w_i)$ exploiting the BCH formula \eqref{eq:BCH} with respect to the basis in \eqref{eq:base}:
\begin{align*}
G(d\phi_n)(w_i)=&w_i L(w_i)=w_i+L(w_i)\\
&+\sum_{s=2}^\iota \Big( \sum_{j=h_{s-1}+\ell_s+1}^{h_s} Q_j\big(\bar{\cp}_{s-1}(w_i),\bar{\cp}_{s-1}(L(w_i))\big)w_j \stepcounter{equation}\tag{\theequation}\label{eq:bch}\\
& \qquad \qquad \qquad \qquad  + \sum_{j=h_{s-1}+1}^{h_{s-1}+\ell_s} Q_j\big(\bar{\cp}_{s-1}(w_i),\bar{\cp}_{s-1}(L(w_i))\big)v_j \Big),
\end{align*}
where the projections $\bar{\cp}_j$ are the projections in \eqref{eq: projections} and $Q_j$ are polynomial on $\G$.
Now, let us observe that if $w_i \in V_k$, for $k=1, \dots , \iota$, then
$$ \bar{\cp}_{s-1}(w_i) =0 $$
for every $s \leq k$, hence
\begin{equation}\label{eq:qj}
Q_j\big(\bar{\cp}_{s-1}(w_i),\bar{\cp}_{s-1}(L(w_i))= Q_j\big(0,\bar{\cp}_{s-1}(L(w_i))=0
\end{equation}
for every $j \in \{h_{s-1}+1, \cdots, h_s \}$ and $2\le s \leq k$.
Thus, by \eqref{eq:bch} and \eqref{eq:qj} we deduce that if $w_i \in V_k$
\begin{equation}\label{eq:projvect}
\begin{split}
P_{\W}(G(d\phi_n)(w_i))&=w_i+\sum_{s=2}^\iota  \sum_{j=h_{s-1}+\ell_s+1}^{h_s} Q_j\big(\bar{\cp}_{s-1}(w_i),\bar{\cp}_{s-1}(L(w_i))\big)w_j\\
&=w_i+\sum_{s=k+1}^\iota  \sum_{j=h_{s-1}+\ell_s+1}^{h_s} Q_j\big(\bar{\cp}_{s-1}(w_i),\bar{\cp}_{s-1}(L(w_i))\big)w_j.
\end{split}
\end{equation}
Notice in particular that if $k=\iota$, namely $w_i \in V_{\iota}$, then $P_{\W}(G(d\phi_n)(w_i))=w_i$,
therefore the form of the projected vectors in \eqref{eq:projvect} gives
\begin{align*}
\bigwedge_{i \in I_{\W}} P_{\W}(G(d\phi_n)(w_{i}))=\bigwedge_{i \in I_{\W}}w_i
\end{align*}
and this concludes the proof of the proposition.
\end{proof}

\subsection{Spherical factor and upper blow-up}
The spherical measure in the area formula needs a renormalization, given by a geometric factor. This number clearly depends on the fixed homogeneous distance $d$ and the Euclidean norm $|\cdot|$ 
arising from the fixed scalar product $\lan\cdot,\cdot\ran$ on $\G$. 

\begin{deff}[Spherical factor]
	\label{def:sphfactor}
	Let $\mathbb{W}\subset\G$ be a linear subspace of topological dimension $n$. We define the \emph{spherical factor of $d$ with respect to 
	$\mathbb{W}$} as the number
	$$ \beta_d (\mathbb{W} )= \max_{z \in \mathbb{B}(0,1)} \mathcal{H}^n_{|\cdot|} ( \mathbb{W} \cap \mathbb{B}(z,1)),$$
	where the Hausdorff measure $\cH^n_{|\cdot|}$ is constructed by the fixed Euclidean norm $|\cdot|$.
\end{deff}

 \begin{teo}[Upper blow-up]\label{theorem:ubu}
Let $(\W,\V)$ be a couple of complementary subgroups of $\G$. Let $m$ and $M$ be the topological and the Hausdorff dimensions of $\W$, respectively. We consider an open set $A \subset \W$ and $\phi:A \to \V$. We also assume that $\phi$ is intrinsically differentiable at any point of $A$ and that $d\phi:A \to \IL(\W,\V)$ is  continuous.
We set $\U_{w}=\mathrm{graph}(d\phi_{w})$ for every $w \in A$. Let $\Phi:A\to\G$ be the graph map of $\phi$ and introduce the following measure
\begin{equation}
\mu(B)= \int_{\Phi^{-1}(B)} J\Phi(w)\  d \mathcal{H}_{|\cdot|}^m (w)
\end{equation}
for every Borel set $B \subset \mathbb{G}$. 
Setting $\Sigma=\Phi(A)$, for every $x=\Phi(\zeta) \in \Sigma$ we have 
$$
\cs^M( \mu, x)= \ \beta_{d}( \U_{\zeta}).
$$
\end{teo}

\begin{proof}
Let us consider $x=\Phi(\zeta) \in \Sigma$ with $\zeta \in A$.
For any $y \in \G$ and $t>0$, we can write
\begin{equation} 
\mu(\mathbb{B}(y,t))= 
 \int_{\Phi^{-1}(\mathbb{B}(y,t))} J \Phi(w) \  d \mathcal{H}_{| \cdot |}^{m} (w).
\end{equation}
We now perform the change of variables
\[
w= \sigma_x(\Lambda_t(\eta)),
\]
where $\Lambda_t= \delta_t|_{\mathbb{W}}$. The Jacobian of $\Lambda_t$ is $ t^M$ and $\sigma_x$ is measure-preserving (\cite[Lemma 2.20]{FranchiSerapioni2016IntrLip}).
We obtain that
\[
\frac{\mu(\mathbb{B}(y,t))}{t^M} =
  \int_{\Lambda_{1/t}(\sigma_x^{-1}(\Phi^{-1}(\mathbb{B}(y,t))))} 
J\Phi(\sigma_x (\Lambda_t(\eta))))  \  d \mathcal{H}^m_{|\cdot|} (\eta).
\]
By the definition of spherical Federer density, we get
\begin{align*}
\cs^M(\mu ,x)& = \inf_{r>0} \sup_{\substack{y \in \mathbb{B}(x,t)\\ 0<t<r}} \frac{\mu (\mathbb{B}(y,t))}{t^{M}} \\
&=  \inf_{r>0} \sup_{\substack{ y \in \mathbb{B}(x,t) \\ 0<t<r}}  \ \int_{\Lambda_{1/t}(\sigma_x^{-1}(\Phi^{-1}(\mathbb{B}(y,t))))} 
J\Phi (\sigma_x(\Lambda_t(\eta)) ) \  d \mathcal{H}^{m}_{|\cdot|} (\eta).
\end{align*}
There exists $R_0>0$ such that for $t>0$ and $y \in \mathbb{B}(x,t)$ 
we have the following inclusion
\begin{eqnarray}\label{insieme}
\Lambda_{1/t}(\sigma_x^{-1}(\Phi^{-1}(\mathbb{B}(y,t)))) \subset \mathbb{B}_{\mathbb{W}}(0,R_0),
\end{eqnarray}
where the translated function $\phi_{x^{-1}}$ is defined according to formula \eqref{eq:translatedf} and
we have set
\[
\mathbb{B}_{\mathbb{W}}(0,R_0)=\mathbb{B}(0,R_0)\cap \mathbb{W}.
\]
To see \eqref{insieme}, we write more explicitly $\Lambda_{1/t}(\sigma_x^{-1}(\Phi^{-1}(\mathbb{B}(y,t))))$,
that is 
\[
\left\{ \eta \in \Lambda_{1/t}(\sigma_x^{-1}(A)) : \left\|y^{-1} \Phi(\sigma_x(\Lambda_t(\eta))) \right\|\le t \right\}. 
\]
It can be written as follows
\[
\left\{ \eta \in \Lambda_{1/t}(\sigma_x^{-1}(A)) : \left\|y^{-1} \sigma_x(\Lambda_t(\eta))\phi(\sigma_x(\Lambda_t(\eta))) \right\|\le t \right\},
\]
hence we can continue getting
\[
\left\{ \eta \in \Lambda_{1/t}(\sigma_x^{-1}(A)) : \|y^{-1} x \Lambda_t(\eta) (\pi_{\V}(x \Lambda_t(\eta))^{-1}\phi(\sigma_x(\Lambda_t(\eta)))\| \leq t \right\}. 
\]
According to \eqref{eq:translatedf}, the translated function of $\phi$ at $x^{-1}$ is 
\[
\phi_{x^{-1}}(\eta)=(\pi_\V(x\eta))^{-1}\phi(\sigma_x(\eta))
\]
Thus, we finally get 
\beq\label{eq:setLambda}
\begin{split}
\Lambda_{1/t}&(\sigma_x^{-1}(\Phi^{-1}(\mathbb{B}(y,t)))) =
\left\{ \eta \in \Lambda_{1/t}(\sigma_x^{-1}(A)) : \|y^{-1} x \Lambda_t(\eta) \phi_{x^{-1}}( \Lambda_t(\eta))\| \leq t \right\}\\
&= \left\{ \eta \in \Lambda_{1/t}(\sigma_x^{-1}(A)) : 
\left\|(\delta_{1/t}(y^{-1}  x ))  \eta \left(\delta_{1/t}(\phi_{x^{-1}}(\Lambda_t \eta))\right)\right\|\le1 \right\},
\end{split}
\eeq
hence for $\eta \in \Lambda_{1/t}(\sigma_x^{-1}(\Phi^{-1}(\mathbb{B}(y,t)))$, taking into account 
the previous equality, we have established that
$$ \eta\left(\delta_{1/t}(\phi_{x^{-1}}(\Lambda_t \eta)) \right) \in \mathbb{B}(0,2).$$
From the estimate \eqref{eqczero}, we know that
$$c_0 \left( \| \eta \| + \left\Vert   \delta_{1/t}(\phi_{x^{-1}}(\Lambda_t\eta)) \right\Vert \right)\leq 
\left\Vert \eta \left(\delta_{1/t}(\phi_{x^{-1}}(\Lambda_t\eta))\right) \right\Vert \leq 2,$$
hence the inclusion \eqref{insieme} holds with $R_0=2/c_0$. As a consequence, we have that $$ \cs^{M}(\mu ,x) < \infty.$$
There exist a positive sequence $ \{t_k \}$ converging to zero and $y_k \in \mathbb{B}(x,t_k)$ such that
\[
 \int_{\Lambda_{1/t_k}(\sigma_x^{-1}(\Phi^{-1}(\mathbb{B}(y_k,t_k))))} J \Phi (\sigma_x(\Lambda_{t_k} (\eta) ) \ d  \mathcal{H}^{m}_{|\cdot|}(\eta)\to \cs^{M}(\mu,x)
\]
as $k\to\infty$. Up to extracting a subsequence, since $y_k \in \mathbb{B}(x,t_k)$ for every $k$, there exists $z \in \mathbb{B}(0,1)$ such that
$$ \lim_{k \to \infty} \delta_{1/t_k}(x^{-1}  y_k) = z.$$
We set
\[
S_z=\pi_{\W,\U_{\zeta}}^{\W,\V}(\U_{\zeta}\cap \mathbb{B}(z,1))\subset\W.
\]
\textbf{Claim 1:} For each $\omega \in \mathbb{W}\sm S_z$, there exists
$$
\lim_{k \to \infty} 1_{\Lambda_{1/t_k}(\sigma_x^{-1}(\Phi^{-1}(\mathbb{B}(y_k,t_k)))} (\omega)=0.
$$
By contradiction, if we had a subsequence of the integers $k$ such that 
\begin{equation*}
(\delta_{1/t_k}(y_k^{-1}  x )) \omega \pa{\delta_{1/t_k}(\phi_{x^{-1}}(\Lambda_{t_k} \omega)) }\in \mathbb{B}(0,1),
\end{equation*}
then by a slight abuse of notation, we could still call $t_k$ the sequence such that
\beq\label{eq:inclusion}
(\delta_{1/t_k}(y_k^{-1}  x ))  \omega    d \phi_{\zeta}(\omega)     \pa{\delta_{1/t_k}( (d \phi_{\zeta}(\Lambda_{t_k} \omega))^{-1} \phi_{x^{-1}}(\Lambda_{t_k} \omega))}  \in \mathbb{B}(0,1)
\eeq
for all $k$, where we exploited the intrinsic differentiability of $\phi$ at $\zeta$ and the homogeneity of the intrinsic differential $d\phi_\zeta$.
Due to the intrinsic differentiability, taking into account \eqref{eq:inclusion} as $k\to\infty$,
it follows that
$$
\omega d \phi_{\zeta}(\omega) \in \mathbb{B}(z, 1).
$$
As a consequence, $ \omega d \phi_{\zeta}(\omega) \in \mathbb{B}(z,1) \cap \U_\zeta$
and then 
\beq\label{eq:piS_z}
\omega=\pi_{\W,\U_\zeta}^{\W,\V}(\omega d \phi_{\zeta}(\omega))\in \pi_{\W,\U_\zeta}^{\W,\V}(\U_\zeta \cap \mathbb{B}(z,1))=S_z,
\eeq
that is not possible by our assumption. This concludes the proof of Claim 1.

Now we introduce the density function
$$
\alpha (t, \eta)=J \Phi(\sigma_x(\Lambda_t(\eta)))
$$
to write 
$$
 \int_{\Lambda_{1/t_k}( \sigma_x^{-1}(\Phi^{-1}(\mathbb{B}(y_k,t_k))))} \alpha(t_k, \eta) \ d   \mathcal{H}^{m}_{|\cdot|} (\eta) = I_k+J_k.
$$
The sequence $I_k$, defined in the following equality, satisfies the estimate
\[
I_k=  \  \int_{S_z \cap \Lambda_{1/t_k}(\sigma_x^{-1}(\Phi^{-1}(\mathbb{B}(y_k,t_k))))} \alpha(t_k, \eta) \ d   \mathcal{H}^{m}_{|\cdot|} (\eta) 
\leq  \int_{S_z } \alpha(t_k, \eta) \ d   \mathcal{H}^{m}_{|\cdot|} (\eta).
\]
Analogously for   
\[
J_k=  \ \int_{\Lambda_{1/t_k}(\sigma_x^{-1}(\Phi^{-1}(\mathbb{B}(y_k,t_k)))) \sm S_z} \ \alpha(t_k, \eta) \ d   \mathcal{H}^{m}_{|\cdot|} (\eta),
\]
we get 
\[
J_k \leq   \int_{\mathbb{B}_{\mathbb{W}}(0,R_0) \sm S_z} 1_{\Lambda_{1/t_k}(\sigma_x^{-1}((\Phi^{-1}(\mathbb{B}(y_k,t_k))))}(\eta) \ \alpha(t_k, \eta) \ d   \mathcal{H}^{m}_{|\cdot|} (\eta).
\]
Claim 1 joined with the dominated convergence theorem proves that $J_k\to0$ as $k\to \infty$,
hence $I_k\to\cs^{M}(\mu,x)$. To study the asymptotic behaviour of $I_k$, we
first observe that the continuity of $J\Phi$ ensured by Proposition~\ref{cont-jac} yields
\[
\alpha(t_k,\eta) \to J\Phi(\zeta)
\]
as $k \to \infty$. 
It follows that  
\beq\label{eq:theta}
\cs^{M}(\mu,x)=\lim_{k\to\infty}I_k\leq   J\Phi(\zeta) \ \mathcal{H}^{m}_{|\cdot|}(S_z).
\eeq
As a result, taking into account \eqref{eq:theta} and Proposition~\ref{propjac}, we have proved that
\beq
\cs^{M}(\mu ,x)  \leq \frac{ | \bV \wedge \bW |}{| \bV \wedge \bU_{\zeta} |} \ \mathcal{H}^{m}_{|\cdot|}(S_z)
\eeq
where $\bV$ is an orienting unit $p$-vector of $\V$ and $\bW$, $\bU_{\zeta}$ are orienting unit $(q-p)$-vectors of $\W$ and $\U_{\zeta}$, respectively.
By Lemma~\ref{lm:MVPsubgroups}, for $B=\U_\zeta\cap\B(z,1)$, the following formula holds
\beq\label{eq:LemmaAlg} 
\mathcal{H}^{m}_{|\cdot|}(\pi_{\W, \U_\zeta}^{\W, \V}(\U_\zeta\cap\B(z,1)))= \frac{| \bV \wedge \bU_\zeta |}{| \bV \wedge \bW |} \ \mathcal{H}^{m}_{|\cdot|}(\U_\zeta\cap\B(z,1)).
\eeq
It follows that 
\beq\label{eq:firstIneq}
\cs^{M}(\mu ,x) \leq\mathcal{H}^{m}_{|\cdot|}(\U_\zeta \cap \mathbb{B}(z,1)) \leq   \mathcal{H}^{m}_{|\cdot|}(\U_\zeta \cap \mathbb{B}(z_0,1)),
\eeq
where $z_0 \in \mathbb{B}(0,1)$ is chosen
such that $\beta_d(\U_\zeta)= \mathcal{H}^{m}_{|\cdot|}(\U_\zeta \cap \mathbb{B}(z_0,1) )$.

For the opposite inequality, we follow the scheme in the proof of \cite[Theorem 3.1]{Mag31}.
We consider a specific family of points 
$y_t^0= x \delta_t z_0 \in \mathbb{B}(x,t)$ and fix $\lambda>1$. We have that
$$
\sup_{0<t<r} \frac{\mu(\mathbb{B}(y_t^0,\lambda t))}{(\lambda t)^{M}} \leq \sup_{\substack{y \in \mathbb{B}(x,t), \\ 0 < t < \lambda r}} \frac{\mu(\mathbb{B}(y,t))}{t^{M}}
$$
for every $r>0$, therefore
\beq\label{eq:limsuptheta}
\limsup_{t \to 0^+} \frac{\mu(\mathbb{B}(y_t^0,\lambda t) )}{(\lambda t)^{M}} \leq \cs^{M}(\mu ,x).
\eeq
We introduce the set 
\begin{align*}
A^0_t &= \Lambda_{1/\lambda t}(\sigma_x^{-1}(\Phi^{-1}(\mathbb{B}(y_t^0, \lambda t)))\\
& = \left\{ \eta \in \Lambda_{1/\lambda t} (\sigma_{x}^{-1}(A)) :  \eta  \left( \delta_{\frac{1}{\lambda t}}(\phi_{{x}^{-1}}(\Lambda_{\lambda t} \eta)) \right) \in \mathbb{B}(\delta_{1/ \lambda} z_0, 1)    \right\}.
\end{align*}
The second equality can be deduced from \eqref{eq:setLambda}. Then we can rewrite
\begin{equation}\label{eq:mu}
\frac{\mu(\mathbb{B}(y_t^0,\lambda t))}{(\lambda t)^M}= \int_{A_t^0} \alpha( \lambda t, \eta) d \mathcal{H}^{m}_{|\cdot|}( \eta) 
= \frac{1}{\lambda^M} \int_{\delta_{\lambda} A_t^0} \alpha( \lambda t, \delta_{1/\lambda} \eta) d \mathcal{H}^{m}_{|\cdot|}(\eta).
\end{equation}
The domain of integration satisfies
$$ \delta_{\lambda} A_t^0= \left\{ \eta \in \Lambda_{1/t}(\sigma_{x}^{-1}(A)) : \eta \left( \delta_{1/t}(\phi_{{x}^{-1}}(\Lambda_{ t} \eta)) \right) \in \mathbb{B}(z_0, \lambda) \right\}.$$
Due to \eqref{insieme} and the definition of $A^0_t$, it holds  
\[
\delta_{\lambda} A_t^0 \subset \mathbb{B}_\W(0,\lambda R_0).
\]
\textbf{Claim 2:} For every $\eta \in \pi_{\W, \U_\zeta}^{\W, \V}( \U_\zeta \cap B(z_0, \lambda))$, we have
\beq\label{eq:limA^0=1}
\lim_{t \to 0^+}1_{\delta_{\lambda} A_t^0} (\eta)=1.
\eeq	
The intrinsic differentiability of $\phi$ at $\zeta$ shows that
\[
\eta \left( \delta_{1/t}(\phi_{{x}^{-1}}(\Lambda_{ t} \eta)) \right)\to \eta d\phi_\zeta(\eta)\quad\text{as}\quad t\to0.
\]
Taking into account \eqref{eq:invproj} and \eqref{eq:piS_z}, we get
\[
\pi^{\U_\zeta,\V}_{\U_\zeta,\W}(\eta)=\eta d\phi_\zeta(\eta),
\]
hence our assumption on $\eta$ can be written as follows
\[
d\pa{\eta d\phi_\zeta(\eta),z_0}<\lambda.
\]
We conclude that $\eta\in \delta_{\lambda} A_t^0$ for any $t>0$ sufficiently small, therefore the limit \eqref{eq:limA^0=1} holds and the proof of Claim 2 is complete.

By Fatou's lemma, taking into account \eqref{eq:limsuptheta} and \eqref{eq:mu}
we get
\[
\frac{1}{\lambda^{M}} 
\int_{\pi_{\W, \U_{\zeta}}^{\W, \V}( \U_{\zeta} \cap B(z_0, \lambda))} \liminf_{t\to0}\pa{1_{\delta_{\lambda} A_t^0}(\eta) \alpha( \lambda t, \delta_{1/\lambda} \eta)} d \mathcal{H}^{m}_{|\cdot|}(\eta)\le \cs^{M}(\mu ,x).
\]
Claim 2 joined with Propositions~\ref{cont-jac} and \ref{propjac} yield
\[
\frac1{\lambda^{M}} \frac{| \bV \wedge \bW |}{| \bV \wedge \bU_\zeta |}
\mathcal{H}^{m}_{|\cdot|}\pa{\pi_{\W, \U_\zeta}^{\W, \V}( \U_\zeta \cap\B(z_0,1))}\le \cs^{M}(\mu ,x).
\]
Applying again \eqref{eq:LemmaAlg}, we obtain
\[
\frac1{\lambda^{M}} \mathcal{H}^{m}_{|\cdot|}(\U_\zeta\cap\B(z_0,1))\le \cs^{M}(\mu ,x).
\]
Taking the limit as $\lambda\to1^+$ and considering the opposite inequality \eqref{eq:firstIneq}, the proof is complete.
\end{proof}

The previous theorem joined with Theorem~\ref{thm:metarea} leads us to a general area formula for graphs of intrinsically differentiable maps with continuous intrinsic differential.

\begin{teo}[Area formula]\label{teo:area}
	We consider a couple $(\W,\V)$ of complementary subgroups of $\G$. Let $m$ and $M$ be the topological and the Hausdorff dimensions of $\W$, respectively. We consider an open set $A \subset \W$ and a mapping $\phi:A \to \V$. We also assume that $\phi$ is intrinsically differentiable at any point of $A$ and that $d\phi:A \to \IL(\W,\V)$ is  continuous.
	Setting $\Sigma=\Phi(A)$, for every Borel set $B \subset \Sigma $ we have
	\begin{equation}\label{eq:area}
		\int_{\Phi^{-1}(B)} J\Phi(w)\  d \mathcal{H}_{|\cdot|}^{m} (w)= \int_B \beta_{d}(\T_x) \ d \mathcal{S}^{M}(x),
	\end{equation}
	where $\T_x$ is the tangent subgroup to $\Sigma$ at $x$. 
\end{teo}
\begin{proof}
	Defining the measure
	\[
	\mu(B)=\int_{\Phi^{-1}(B)} J\Phi(w)\  d \mathcal{H}_{|\cdot|}^{m} (w)
	\]
	for every Borel set $B\subset \G$, we are in the assumption of Theorem~\ref{theorem:ubu},
	hence we get 
	\beq \label{eq:Fedbeta}
	\cs^M( \mu, x)= \ \beta_{d}( \graph(d\phi_\zeta))
	\eeq
	for every $x=\Phi(w)\in\Sigma$, where $\cs^M(\mu,x)$ is the spherical Federer $M$-density of $\mu$ at $x$.
	Due to Theorem~\ref{characterizationP}, the set $\graph(\phi)$ has a tangent subgroup $\T_x$ at $x$
	which equals $\graph(d\phi_w)$, therefore $\cs^M( \mu, x)= \ \beta_{d}( \T_x)$.
	To conclude the proof we apply Theorem~\ref{thm:metarea}.
	First of all, the measure $\mu$ can be immediately extended to an outer measure on $\G$, that is automatically Borel regular, so condition (1) of Theorem~\ref{thm:metarea} is proved.
	The diametric regularity of homogeneous groups and the
	validity of condition (2) of the same theorem are discussed
	in Section~\ref{sect:FedDens}. By standard arguments of elementary topology one can cover $\Sigma$ with a countable family of metric balls with $\mu$ finite measure, hence condition (3) of
	Theorem~\ref{thm:metarea} holds. 
	Since the spherical factor \eqref{eq:Fedbeta} is everywhere finite and positive on $\Sigma$, then both conditions (4) and (5) of Theorem~\ref{thm:metarea} hold.
	The latter follows from \cite[Proposition~3.3]{LecMag22}.  
	Finally, the application of \eqref{eq:metarea} concludes the proof.
\end{proof}

An important consequence is the following result.

\begin{teo}\label{teo:areauid}
Let $(\W,\V)$ be a couple of complementary subgroups of $\G$. Let $m$ and $M$ be the topological and the Hausdorff dimension of $\W$, respectively. We consider an open set $A \subset \W$, $\phi:A \to \V$ and  define $\Sigma=\Phi(A)$. 
If $\phi$ is uniformly intrinsically differentiable on $A$, then for every Borel set $B \subset \Sigma $ we have
	\begin{equation}
		\int_{\Phi^{-1}(B)} J\Phi(w)\  d \mathcal{H}_{|\cdot|}^{m} (w)= \int_B \beta_{d}(\T_x) \ d \mathcal{S}^{M}(x),
	\end{equation}
	where $\T_x$ is the tangent subgroup to $\Sigma$ at $x$, according to Definition~\ref{TangenteP}.
\end{teo}
\begin{proof}
	It is sufficient to observe that $\phi$ satisfies the hypothesis of Theorem~\ref{teo:area}. By definition of uniform intrinsic differentiability on $A$, $\phi$ is intrinsically differentiable at every point of $A$. The continuity of the intrinsic differential $A \ni w \to d\phi_w \in \IL(\W,\V)$ follows from Proposition~\ref{uidimpliescont}.
\end{proof}

\section{Applications to special classes of intrinsic graphs}\label{sect:app}

This section is divided into two parts.
The first one introduces the area formula for the level sets arising from Theorem~\ref{corintdiffparam}, hence proving Theorem~\ref{teo:arealevelset}.
The second part provides an area formula for 
all $(\G,\R^k)$-regular sets of $\G$, proving that the Jacobian \eqref{def:JacPhi} can be written in terms of suitable partial derivatives of the parametrization.

\subsection{Area formula for level sets}\label{sect:arealevels}

The aim of the section is to prove Theorem~\ref{teo:arealevelset}, which is a special version of Theorem~\ref{teo:areaIntro}. Precisely, we wish to compute the spherical measure of 
$(\G,\M)$-regular sets of $\G$ (Definition~\ref{def:GMregular}). $\G$ and $\M$ are assumed to be stratified groups.

To find the area of $(\G,\M)$-regular sets of $\G$, we use their graph structure and the general area formula of Theorem~\ref{teo:areauid}.
Let $\Omega \subset \G$ be an open set and let $f \in C^1_h(\Omega, \M)$. We consider an orthonormal basis $(b_1, \dots, b_p)$ of $\M$, along with its dual basis $(b_1^{\star}, \dots, b_p^{\star}) \subset \M^{\star}$. For every $i=1, \dots p$, we denote by $R^i_Hf(x)$ the unique vector that represents the linear map $b_i^{\star} \circ Df(x):\G \to \R$ as follows
\[
(b^{\star}_i \circ Df(x))(v)= \langle R^i_Hf(x), v \rangle 
\]
for every $v \in \G$. For a homogeneous subgroup $\V \subset \G$ and for $i=1 \dots, p$, the vector $R^i_{\V}f(x) \in \V$ represents the linear map $b_i^{\star} \circ Df(x)|_{\V}:\V \to \R$ as follows 
\[
(b_i^{\star} \circ Df(x)|_{\V})(v)= \langle R^i_{\V}f(x),v \rangle 
\]
for every $v \in \V$. Notice that if $(b_1, \dots, b_p)$ is an orthonormal basis, then a standard argument (see for instance \cite[Section~2.1]{GMS98}) ensures that
\begin{equation}
	J_Hf(x)= | R^1_Hf(x) \wedge \dots \wedge R^p_Hf(x) | \quad \mathrm{and} \quad J_{\V}f(x)= | R^1_{\V}f(x)\wedge \dots \wedge R^p_{\V}f(x)|.
\end{equation}

We conclude the section with the proof of the area formula for $(\G,\M)$-regular sets. 

\begin{proof}[Proof of Theorem~\ref{teo:arealevelset}]
	Remark~\ref{rem:exglobA} ensures the existence and uniqueness of the mapping $\phi:A\to\V$ such that $\Sigma\cap\Omega'=\Phi(A)$.
	By Theorem~\ref{intdiffpar}, the map $\phi$ is uniformly intrinsically differentiable on $A$. Let us consider $x =\Phi(\zeta) \in \Sigma\cap\Omega'$ and set $\U_\zeta=\mathrm{graph}(d \phi_{\zeta})$.
	By Theorem~\ref{intdiffpar}, we obtain 
	$$
	\T_x=\U_{\zeta}=  \ker (Df(x)),
	$$
	where $\T_x$ is the tangent subgroup to $\Sigma$ at $x$, according 
	to Definition~\ref{TangenteP}.
	Thus, by exploiting Theorem~\ref{teo:areauid} we get 
	\begin{equation}\label{formula-phi}
		\int_B \beta_{d}(\T_x) \ d \mathcal{S}^{Q-P}(x)= \int_{\Phi^{-1}(B)} J \Phi(w) \  d \mathcal{H}_{|\cdot|}^{q-p} (w)
	\end{equation}
	for every Borel set $B \subset\Sigma\cap \Omega'$.
	Consider $\bU_\zeta= u_{p+1} \wedge \dots \wedge u_{q} $ such that $ ( u_{p+1}, \dots, u_{q} )$ is an orthonormal basis of $\U_\zeta$. We claim that
	\beq\label{eq:c(x)}
	\frac{ J_Hf(x)}{J_{\V}f(x)} = \frac{1}{| \bV \wedge \bU_\zeta |}.
	\eeq
	Let $(v_1, \dots, v_p)$ be an orthonormal basis of $\V$ and let us complete it to an orthonormal basis $(v_1, \dots, v_q)$ of $\G$.
	Let us denote by $*$  the Hodge operator in $\G$ with respect to the fixed orientation $$\textbf{e}= v_1 \wedge \dots \wedge v_{q} .$$ For a $p$-vector $\eta\in\Lambda_p\G$, the $(q-p)$-vector $ \ast \eta$ is uniquely defined by
	\begin{equation}\label{eq10}
		\xi \wedge \ast \eta= \langle \xi, \eta \rangle \textbf{e}
	\end{equation}
	for all $p$-vectors $\xi \in \Lambda_p\G$.
	We notice that $\spn \{ R_H^1(x), \dots, R_H^pf(x) \}$
	is orthogonal to $\U_{\zeta}$. In fact for every $i \in \{1, \cdots, p \}$ and $j \in \{ p+1, \cdots,  q\}$ we have
	\begin{align*}
		\langle R_H^if(x), u_j \rangle= (b_i^{\star} \circ Df(x))(u_j)=b_i^{\star}(Df(x)(u_j))=0.
	\end{align*}
	It follows that
	\begin{equation}\label{ker2}
		u_{p+1} \wedge \dots \wedge u_{q} = \ast ( R_H^1f(x) \wedge \dots \wedge R^p_Hf(x)) \lambda
	\end{equation}
	for some $\lambda \in \mathbb{R}$ (see for instance {\cite[Lemma 5.1]{Mag12A}}). Since the Hodge operator is an isometry, we get
	\begin{equation}\label{eq9}
		| \lambda |= \frac{1}{| R_H^1f(x) \wedge \dots \wedge R^p_Hf(x) |}.
	\end{equation}
	Due to \eqref{eq9} and \eqref{eq10}, we have
	\begin{align*}
		| \bV \wedge \bU_{\zeta} | &= | \lambda| \,
		| v_1 \wedge \dots \wedge v_p \wedge \ast ( R_H^1f(x) \wedge \dots \wedge R^p_Hf(x)) | \\
		&= \frac{| \langle v_1 \wedge \dots \wedge v_p, R_H^1f(x) \wedge \dots \wedge R^p_Hf(x) \rangle \textbf{e} |}{| R_H^1f(x) \wedge \dots \wedge R^p_Hf(x)|  }\\
		&= \frac{| \langle v_1 \wedge \dots \wedge v_p, R_H^1f(x) \wedge \dots \wedge R^p_Hf(x) \rangle|}{ |  R_H^1f(x) \wedge \dots \wedge R^p_Hf(x)| }\\
		&= \frac{| R_{\V}^1f(x) \wedge \dots \wedge R^p_{\V}f(x) |}{ | R_H^1f(x) \wedge \dots \wedge R^p_Hf(x) |  }=\frac{J_{\V}f(x)}{J_Hf(x)},
	\end{align*}
	hence \eqref{eq:c(x)} is proved. Thus, by combining \eqref{eq:c(x)} and Proposition~\ref{propjac}, surely
	\beq\label{eq:duejac}
	J\Phi(\zeta) =  \frac{|\bV \wedge \bW|}{| \bV \wedge \bU_\zeta |}= |\bV \wedge \bW| \ \frac{ J_Hf(x)}{J_{\V}f(x)} 
	\eeq
	for every $x=\Phi(\zeta) \in \Sigma\cap\Omega'$. Combining \eqref{eq:duejac} and \eqref{formula-phi}, we get \eqref{formula-f}, concluding the proof.
\end{proof}

\begin{Remark}
	The previous theorem gives a local formula for the spherical measure, that can be globally defined on any $(\G,\M)$-regular set.
\end{Remark}

\subsection{Area formula and intrinsic partial derivatives}\label{sect:intder}

In this section, we assume that $(\W,\V)$ is a couple of \textit{orthogonal} complementary subgroups of a stratified group $\G$. Moreover, we assume that $\V$ is a \textit{horizontal subgroup}, i.e.\ $\V$ is a homogeneous subgroup contained in $V_1$. Under these assumptions, we obtain a more explicit form of the Jacobian, in terms of suitable intrinsic partial derivatives.

Now we fix some notation that will be used throughout the section. We denote by $p$ the topological dimension of $\V$ and we consider $(v_1, \dots, v_p, w_{p+1}, \dots, w_q)$ an orthonormal  graded basis of $\G$ such that $(v_1, \dots, v_p)$ is a basis of $\V$ and $(w_{p+1}, \dots, w_q)$ is a basis of $\W$. 
We introduce the maps
\begin{align}
	i_{\V}:\V \to \R^p,& \ i_{\V} \Bigg( \sum_{i=1}^p x_iv_i \Bigg)=(x_1, \dots, x_p) \label{i_V} ,\\
	i_{\W}:\W \to \R^{q-p},& \ i_{\W}\Bigg(\sum_{i=p+1}^q x_iw_i \Bigg)=(x_{p+1}, \dots, x_q) \label{i_W} ,\\
	i_{\G}:\G \to \R^{q},& \ i_{\G}\Bigg(\sum_{i=1}^p x_i v_i + \sum_{i=p+1}^q x_i w_i  \Bigg)=(x_{1}, \dots, x_q). \label{i_G}
\end{align}
If $A \subset \W$, we define $\widetilde{A} = i_{\W}(A) \subset \R^{q-p} $ and for $w \in A$, we set $\widetilde{w}=i_{\W}(w) \in \widetilde{A}$.
For $A \subset \W$ and $\phi:A \to \V$, we denote by $\widetilde{\phi}:\widetilde{A} \to \R^p$ the map 
\[
\widetilde{\phi}=i_{\V} \circ \phi \circ i_{\W}^{-1}
\]
and we set its components $\phi_i:\widetilde{A} \to \R$, for $i=1, \ldots, p$ as the maps such that 
\[
\widetilde{\phi}(\widetilde{w})=(\widetilde\phi_1(\widetilde{w}),\widetilde\phi_2(\widetilde{w}), \dots, \widetilde\phi_p(\widetilde{w}))
\]
for every $\widetilde{w} \in \widetilde{A}$. If $\Phi: A \to \G$ is the graph map of $\phi$, we introduce $\widetilde{\Phi}: \widetilde{A} \to \G, \ \widetilde{\Phi}=\Phi \circ i_{\W}^{-1}$.

If $\phi:A \to \V$, with $A \subset \W$ open set, is intrinsically differentiable at $w \in A$, then its intrinsic differential $d\phi_w$ is a linear map, according to \cite[Proposition 3.4]{DiD21}. 

We denote by $\nabla^{\phi}\widetilde\phi(w) \in \R^{p \times (n_1-p)}$ the real matrix representing the intrinsic differential $d\phi_w$ with respect to the bases $(w_{p+1}, \ldots, w_q)$ and $(v_1, \ldots, v_p)$. We denote it equivalently as $\nabla^{\phi}\widetilde\phi(\widetilde{w})$ since it clearly equals the matrix representing $i_{\V} \circ d\phi_w \circ i_{\W}^{-1}: \R^{q-p} \to \R^p$ with respect to the canonical bases of $\R^{q-p}$ and $\R^p$.

For every $j=p+1, \dots q$ we denote by $X_j \in \mathrm{Lie}(\G)$ the left-invariant vector field such that $X_j(0)=w_j$.

\begin{deff}[Projected vector fields]\label{def:projvf}
	Let $A \subset \W$ be an open set and let $\phi:A \to \V$ be a continuous function. Let $\Phi$ be the graph map of $\phi$. For $j=p+1, \dots, q$ we define the \textit{continuous projected vector field $D^{\phi}_{X_j}$} on $\W$ as
	\begin{equation}\label{projvf}
		\big(D^{\phi}_{X_j}\big)_w(f)=\big(X_j \big)_{\Phi(w)}(f \circ \pi_{\W})
	\end{equation}
	for every $w \in A$ and $f \in C^{\infty}(\W)$.
\end{deff}
Notice that by \cite[Remark 3.3]{ADDDLD24}, formula \eqref{projvf} is equivalent to
$$\big(D^{\phi}_{X_j} \big)_w= d (\pi_{\W})_{\Phi(w)}\Big(\big(X_j \big)_{\Phi(w)}\Big).$$

\begin{deff}[Intrinsic partial derivatives]
	\label{intrinsicder}
	Let $A \subset \W$ be an open set, let $\phi:A \to \V$ be a continuous function and consider $w\in A$. Given $j \in \{ p+1, \dots, n_1 \}$, we say that \textit{$\widetilde{\phi}$ has $D^{\phi}_{X_j}$-derivative at $\widetilde{w}$} if and only if there exists a vector $\begin{pmatrix}
		\alpha_{1,j} &
		\dots &
		\alpha_{p,j}
	\end{pmatrix} \in \R^p$ such that for any integral curve $\widetilde{\gamma}: (-\delta, \delta) \to \widetilde{A}$ of $ (i_{\W})_*(D^{\phi}_{X_j})$ with $\widetilde{\gamma}(0)=\widetilde{w}$ the equality
	$$\lim_{s \to 0} \frac{\widetilde{\phi}( \widetilde{\gamma}(s))- \widetilde{\phi}(\widetilde{w})}{s} = \begin{pmatrix}
		\alpha_{1,j} &
		\dots &
		\alpha_{p,j}
	\end{pmatrix}^T$$
	holds.
	We introduce the notation $$
	D^{\phi}_{X_j} \widetilde{\phi} (\widetilde{w})  = 
	\begin{pmatrix}
		D^{\phi}_{X_j} \widetilde\phi_1(\widetilde{w})\\
		\dots\\
		D^{\phi}_{X_j}  \widetilde\phi_p(\widetilde{w})
	\end{pmatrix}=
	\begin{pmatrix}
		\alpha_{1,j}\\
		\dots\\
		\alpha_{p,j}
	\end{pmatrix}$$
	for $j= p+1, \dots, n_1$.
\end{deff}
Taking into account both \cite[Proposition 3.27]{ADDDLD24} and \cite[Proposition 3.19]{ADDDLD24}, one easily observes that the continuity of the intrinsic differential 
immediately gives the continuity of the intrinsic partial derivatives and the following proposition holds.
\begin{prop}\label{prop:components}
	Let $A \subset \W$ be an open set and $\phi:A \to \V$ be a continuous function. Assume that $\phi$ is intrinsically differentiable at every $w \in A$ and assume that $d\phi:A \to \IL_{\W,\V}$ is continuous. Then, for every $i \in \{ 1, \dots , p \}, \ j \in \{p+1, \dots, n_1 \}$ and for every $\widetilde{w}\in \widetilde{A}$ there exists the intrinsic partial derivative $D^{\phi}_{X_j}\widetilde\phi_i(\widetilde{w})$ and
	\begin{equation}\label{eq:intpar-diff}
		D^{\phi}_{X_j} \widetilde\phi_i(\widetilde{w})= [\nabla^{\phi}  \widetilde\phi(\widetilde{w})]_{i,j}.
	\end{equation}
	Moreover, the map $D^{\phi}_{X_j}\widetilde\phi_i: \widetilde{A} \to \R$ is continuous.
\end{prop}

\begin{deff}
	\label{intjacobian}
	Let $A \subset \W$ be an open set, let $w \in A$ and consider a map $\phi: A \to \V$ intrinsically differentiable at $w$.
	We introduce the \emph{intrinsic Jacobian of $\widetilde{\phi}$ at $\widetilde{w}$} as
	\begin{equation}\label{eq:Jacphi}
	J^{\phi} \widetilde{\phi} (\widetilde{w})=  \sqrt{ 1 + \sum_{\ell=1}^{\min\{p,n_1-p\}}\sum_{I \in \mathcal{I}_{\ell}}  (M^{\widetilde \phi}_I(\widetilde{w}))^2 },
	\end{equation}
	where $\mathcal{I}_{\ell}$ is the set of multiindexes
	$$ \{ (i_1, \dots, i_{\ell},j_1, \dots, j_{\ell})) \in \mathbb{N}^{2\ell} : p+1 \leq i_1 < i_2 < \dots < i_{\ell} \leq n_1, \ 1 \leq  j_1 < j_2 \dots < j_{\ell} \leq p \}. $$
	We have also introduced the minors
	\begin{equation*}
		M^{\widetilde{\phi}}_I(\widetilde{w}) =  \mathrm{det} \begin{pmatrix} 
			D^{\phi}_{X_{i_1}} \widetilde\phi_{j_1}(\widetilde{w}) & \dots & D^{\phi}_{X_{i_\ell}} \widetilde\phi_{j_1}(\widetilde{w}) \\
			\dots & \dots & \dots \\
			D^{\phi}_{X_{i_1}} \widetilde\phi_{j_{\ell}}(\widetilde{w}) & \dots & D^{\phi}_{X_{i_\ell}} \widetilde\phi_{j_{\ell}}(\widetilde{w}) \\
		\end{pmatrix}.
	\end{equation*}
\end{deff}

The previous definition is motivated by the observation that the Jacobian \eqref{eq:Jacphi} precisely is the Jacobian of the matrix
\begin{equation}\label{f:jacmat2}
	\begin{bmatrix}
		\nabla^{\phi}\widetilde\phi(\widetilde{w})\\ 
		\mathbb{I}_{n_1-p} 
	\end{bmatrix} \in \R^{n_1 \times (n_1-p)},
\end{equation}
taking into account Proposition~\ref{prop:components}.

\begin{teo}\label{teo:intJacobian}
Let $A \subset \W$ be an open set and consider a continuous map $\phi:A \to \V$. Let us assume that $\phi$ is intrinsically differentiable at each point $w \in A$ and suppose that $d\phi:A \to \IL_{\W,\V}$ is continuous. 
Then for every fixed $u\in A$, setting $\tilde u=i_\W(u)$, with  $i_\W$ defined in \eqref{i_W}, we have 
\begin{equation}\label{eq:eqJac}
J\Phi(u)= J^{\phi} \widetilde{\phi}(\tilde{u}).
\end{equation}

\end{teo}
\begin{proof}
We consider the graph map
$$
G(d\phi_u):\W \to \G, \ G(d\phi_{u})(w)=w d\phi_{u}(w).
$$
Setting $w=\sum_{i=p+1}^{q}x_iw_i \in \W$, with $x_i \in \R$, according to 
the BCH formula and to the definition of $\nabla^{\phi}\widetilde\phi(u)$, we get
\begin{align*}
G(d\phi_{u})(w)=&\Big( \sum_{i=p+1}^{q} x_iw_i \Big) \Big( \sum_{i=1}^p \Big( \sum_{j=1}^{h_1-p} [\nabla^{\phi}\widetilde\phi(u)]_{i,j} x_{j+p}\Big)v_i\Big)\\
=& \sum_{i=1}^p \Big(\sum_{j=1}^{h_1-p} [\nabla^{\phi}\widetilde\phi(u)]_{i,j} x_{j+p}\Big)v_i+ \sum_{i=p+1}^{q} x_i w_i \\
&+ \sum_{s=2}^\iota \sum_{j=h_{s-1}+1}^{h_s} Q_j\left(\sum_{i=p+1}^{h_{s-1}} x_i w_i, \sum_{i=1}^p \Big( \sum_{j=1}^{h_1-p} [\nabla^{\phi}\widetilde\phi(u)]_{i,j} x_{j+p}\Big)v_i \right)w_j.
\end{align*}
We wish to compute the Euclidean Jacobian of 
\begin{equation}\label{mapcoord}
\widetilde G(d\phi_u)= i_{\G}  \circ G(d\phi_{u}) \circ i_{\W}^{-1}:\R^{q-p} \to \R^q
\end{equation}
at an arbitrary point $x \in \R^{p-q}$.
Let us consider the Jacobian matrix of the map \eqref{mapcoord}.
From the above expression of $G(d\phi_{\bar w})$, such Jacobian matrix is of the following form
\begin{equation}\label{jacmat}
\begin{bmatrix}
\nabla^{\phi}\widetilde\phi(u) & 0_{p,n_2}  & 0_{p,n_3} & \ldots & 0_{p,n_{\iota}}\\
\mathbb{I}_{n_1-p} & 0_{n_1-p,n_2} & 0_{n_1-p,n_3} & \ldots & 0_{n_1-p,n_{\iota}}\\
\star & \mathbb{I}_{n_2}& 0_{n_2,n_3} & \ldots & 0_{n_2, n_{\iota}} \\
\ldots & \ldots & \ldots & \ldots & \ldots \\
\star & \star & \star & \mathbb{I}_{n_{\iota}-1} & 0_{n_{\iota}-1,n_{\iota}} \\
\star & \star & \star & \star & \mathbb{I}_{n_{\iota}} \\
\end{bmatrix} \in \R^{q \times (q-p)},
\end{equation}
where by $\star$ we have denoted suitable real matrices, possibly depending on  $x$.
The symbols $0_{i,j} \in \R^{i \times j}$  denote the null matrices and $\mathbb{I}_n$ 
represents the $n \times n$ identity matrix. In the next claim we study some properties of the $(q-p)\times(q-p)$ submatrices of \eqref{jacmat}.\\ 

\textbf{Claim.} If $S \in \R^{(q-p)\times (q-p)}$ is a submatrix  of \eqref{jacmat} such that $\det(S) \neq 0$, 
then $S$ is the square matrix exactly obtained by removing $p$ rows among the first $n_1$ rows of  \eqref{jacmat}.\\

\textit{Proof of Claim}. 
Let us consider the $j$th row of \eqref{jacmat}, which is one of the $p$ rows that we have removed. We wish to show that necessarily $j \leq n_1$.

We argue by contradiction, assuming first that $h_{\iota-1}< j \leq h_{\iota}$. 
Actually we are removing the $(j-p)$th row of the square matrix of \eqref{jacmat} 
made by the last $(q-p)$ rows. This shows that the $(j-p)$th column of $S$ is null, therefore $\det(S)=0$. We have proved that $1\le j \leq h_{\iota-1}$ and the last $n_{\iota}$ rows of $S$ and of \eqref{jacmat} coincide.
Then $S$ is a lower triangular block matrix of the form
\begin{equation}\label{eq:littletriangblock}
S=\begin{bmatrix}
A_{\iota-1} & 0_{q-p-n_{\iota},n_{\iota}}\\
\star & \mathbb{I}_{n_{\iota}}
\end{bmatrix},
\end{equation}
where $A_{\iota-1} \in \R^{(q-p-n_{\iota})\times (q-p-n_{\iota})}$ is a square matrix.
We now assume by contradiction that  $h_{\iota-2}< j \leq h_{\iota-1}$, hence
arguing as before, the special form of \eqref{jacmat} implies that the $(j-p)$th column of $A_{\iota-1}$ is null, therefore the block lower triangular form of \eqref{eq:littletriangblock}  gives $\det(S)=\det(A_{\iota-1})=0$. It follows that $1\le j \leq h_{\iota-2}$ and that the last $n_{\iota-1}+n_{\iota}$ rows of $S$ and of \eqref{jacmat} coincide.

We can iterate the previous arguments, hence getting some integer $2\le s\le\iota-2$
such that $S$ has the following form
\begin{equation}\label{eq:llowtblock}
S=\begin{bmatrix}
A_{s} & 0_{q-p-\sum_{i=s+1}^{\iota}n_i,n_{s+1}}& 0_{q-p-\sum_{i=s+1}^{\iota}n_i,n_{s+2}} & \ldots & 0_{q-p-\sum_{i=s+1}^{\iota}n_i, n_{\iota}} \\
\star & \mathbb{I}_{n_{s+1}}& 0_{n_{s+1},n_{s+2}} & \ldots & 0_{n_{s+1}, n_{\iota}} \\
\ldots & \ldots & \ldots & \ldots & \ldots \\
\star & \star & \star & \mathbb{I}_{n_{\iota-1}} & 0_{n_{\iota-1},n_{\iota}} \\
\star & \star & \star & \star & \mathbb{I}_{n_{\iota}} \\
\end{bmatrix},
\end{equation}
where $A_s \in \R^{(q-p-\sum_{i=s+1}^{\iota}n_i)\times(q-p-\sum_{i=s+1}^{\iota}n_i)}$.
Now, we argue by contradiction, assuming of having removed the $j$th row of \eqref{jacmat}
and $h_{s-1}<j\leq h_{s}$. Thus, as in the previous arguments, 
the block lower triangular form of \eqref{eq:llowtblock} yields 
$\det(S)=\det(A_s)=0$. Therefore, necessarily $j \leq h_{s-1}$. 
Then the argument must stop at $s=2$, therefore 
$1\le j \leq n_1$ and
\begin{equation}\label{eq:llowtblock2}
	S=\begin{bmatrix}
		A_{1} & 0_{q-p-\sum_{i=2}^{\iota}n_i,n_{2}}& 0_{q-p-\sum_{i=2}^{\iota}n_i,n_{3}} & \ldots & 0_{q-p-\sum_{i=2}^{\iota}n_i, n_{\iota}} \\
		\star & \mathbb{I}_{n_2}& 0_{n_{2},n_{3}} & \ldots & 0_{n_{2}, n_{\iota}} \\
		\ldots & \ldots & \ldots & \ldots & \ldots \\
		\star & \star & \star & \mathbb{I}_{n_{\iota-1}} & 0_{n_{\iota-1},n_{\iota}} \\
		\star & \star & \star & \star & \mathbb{I}_{n_{\iota}} \\
	\end{bmatrix},
\end{equation}
finally proving our claim.

Taking into account the block lower triangular form of \eqref{eq:llowtblock2}, according to which 
$\det(S)=\det(A_1)$, the previous claim also proves that there is a bijective correspondence between the
nonvanishing $(q-p)$ minors of \eqref{jacmat} and the nonvanishing $(n_1-p)$ minors of the matrix
\begin{equation}\label{jacmat2}
	\begin{bmatrix}
		\nabla^{\phi}\widetilde\phi(u)\\ 
		\mathbb{I}_{n_1-p} 
	\end{bmatrix} \in \R^{n_1 \times (n_1-p)}.
\end{equation}
As a consequence, the Euclidean Jacobian $J\widetilde G(d\phi_u)(x)$ coincides with the Jacobian of the matrix \eqref{jacmat2}, therefore it is independent of $x \in \R^{p-q}$. Moreover, it is easy to verify that the Jacobian of the matrix \eqref{jacmat2} exactly corresponds to
$J^{\phi} \widetilde{\phi}(\tilde{u})$ introduced in \eqref{eq:Jacphi}. Thus, we have proved that for every $x \in \R^{p-q}$ it holds
\begin{equation}\label{eq:uguaglJac}
J(\widetilde G(d\phi_u))(x)=J^{\phi} \widetilde{\phi}(\tilde{u}).
	\end{equation}
Now, we wish to conclude the proof by showing that 
\begin{equation}\label{eq:uguaglJac2}
	J(\widetilde G(d\phi_u))(x)=J\Phi(u).
\end{equation}
Let us recall the definition of Jacobian in \eqref{def:JacPhi}. Then for every Borel set $B \subset \W$ we have
\begin{equation}
	J\Phi(u)=\frac{\mathcal{H}_{|\cdot|}^{q-p}(G(d\phi_u)(B))}{\mathcal{H}^{q-p}_{|\cdot|}(B)}=\frac{\mathcal{H}_{| \cdot |}^{q-p}(\widetilde G(d\phi_{u})(\widetilde B))}{\mathcal{H}^{q-p}_{| \cdot |}(\widetilde B)},
	\end{equation}
where $\widetilde B=i_\W(B)$ and we have exploited that $i_\W$ and $i_\G$ are isometries. Finally, the Euclidean area formula and
the independence of $J\widetilde G(d\phi_u)(x)$ of $x$ imply \eqref{eq:uguaglJac2}, hence concluding the proof.
\end{proof}

\begin{teo}\label{areaintder}
	Let $A \subset \W$ be an open set and consider a continuous map $\phi:A \to \V$. Assume that $\phi$ is intrinsically differentiable at every $w \in A$ and assume that $d\phi:A \to \IL_{\W,\V}$ is continuous. Set $\Sigma=  \mathrm{graph}(\phi)$.
	Then, for every Borel set $B \subset \Sigma$ we have
	\begin{equation}\label{eq:areaintder}
		\int_{ \widetilde{\Phi}^{-1}(B)} J^{\phi} \widetilde{\phi}(\widetilde{w} )\ d \mathcal{L}^{q-p}(\widetilde{w})=\int_B \beta_{d}(\T_x) \ d \mathcal{S}^{Q-p}(x),
	\end{equation}
	where $\T_x$ is the tangent subgroup to $\Sigma$ at $x$ as in Definition~\ref{TangenteP}.
\end{teo}
\begin{proof}
Combining Theorem \ref{teo:area} and Theorem \ref{teo:intJacobian} we get that for every Borel set $B \subset \Sigma$ the equality
	\begin{equation}
		\int_B \beta_{d}(\T_x) \ d \mathcal{S}^{Q-p}(x)=\int_{\Phi^{-1}(B)} J^{\phi} \widetilde{\phi}(i_\W(w)) \  d \mathcal{H}_{|\cdot|}^{q-p} (w).
	\end{equation}
	holds. Taking into account the equalities 
	\begin{equation}
	(i_\W)_\sharp \cH^{q-p}_{|\cdot|}=\cH^{q-p}_E=\cL^{q-p}
	\end{equation}
	a standard measure-theoretic change of variables gives
	\begin{equation}
		\int_{\Phi^{-1}(B)} J^{\phi} \widetilde{\phi}(i_\W(w)) \  d \mathcal{H}_{|\cdot|}^{q-p} (w)=	\int_{ \widetilde{\Phi}^{-1}(B)} J^{\phi} \widetilde{\phi}(\widetilde{w} )\ d \mathcal{L}^{q-p}(\widetilde{w}),
	\end{equation}
	concluding the proof.
\end{proof}

Combining Theorem~\ref{areaintder} and Theorem~\ref{teo:areasymm} with the results in \cite{Mag22RS,CorMag25} it is not difficult to find important examples of homogeneous distances $d$ for which the area formula \eqref{eq:areaintder} takes the simpler form
\begin{equation}\label{eq:Jconcrete} 
		\mathcal{S}_d^{Q-p} \llcorner \Sigma(B)=\int_{ \widetilde{\Phi}^{-1}(B)} J^{\phi} \widetilde{\phi}(\widetilde{w} )\ d \mathcal{L}^{q-p}(\widetilde{w}).
	\end{equation}
In fact, if $\U \in \mathcal{F}_{\V}$, i.e. $\U$ is a homogeneous subgroup complementary to $\V$, then necessarily it is $(q-p)$-dimensional and there is a linear subspace $U_1 \subset V_1$ such that
\begin{equation}\label{eq:forma}
\U=U_1 \oplus V_2 \oplus \dots \oplus V_{\iota}.
\end{equation}
More generally, it is not difficult to notice that the Cygan--Kor\'anyi distance in H-type groups, \cite{Cygan81}, the distances arising from \cite[Theorem~2]{HebSik90} and the distance $d_\infty$ of \cite[Section~2.1]{FSSC6} are
all multiradial, according to \cite[Definition~1.3]{CorMag25}. Due to \cite[Theorem~3.3]{CorMag25}, it follows that the previous distances 
are rotationally symmetric with respect to the family of homogeneous subgroups $\cF_\V$ of  Theorem~\ref{teo:areasymm} , and in particular \eqref{eq:Jconcrete} follows.
We also have a more general result, \cite[Theorem~1.3]{CorMag25}, that precisely relies on our Theorem~\ref{teo:areasymm}.
Other symmetry results are available if the metric unit ball with respect to $d$ is only assumed to be convex, or if it is $\n$-vertically symmetric, see \cite[Theorem~1.1 and Theorem~1.4]{Mag22RS}.


\end{document}